\documentclass[12pt]{article}
\usepackage{graphicx} 

\usepackage{cmap}  
\usepackage[T2A]{fontenc}
\usepackage[utf8]{inputenc}
\usepackage[english]{babel}
\usepackage[backend=biber,
style=alphabetic, sorting=nyt, isbn=false, date=year]{biblatex}
\AtEveryBibitem{\clearlist{language} \clearfield{urlyear}}

\usepackage[usenames]{color}
\usepackage{amsmath,amssymb,amsthm}
\usepackage[colorlinks=true,linkcolor=blue,urlcolor=red,unicode=true,hyperfootnotes=false,bookmarksnumbered]{hyperref}
\usepackage{mathtools}
\usepackage{xcolor}
\usepackage{csquotes}
\usepackage{wrapfig}
\usepackage{setspace}
\usepackage{enumitem}
\usepackage{appendix}
\usepackage{algorithm}
\usepackage{algpseudocode}
\usepackage{dsfont}
\usepackage{bm}
\usepackage{comment}

\newcommand{\ff}{\mathcal{F}}

\newcommand{\cG}{\mathcal{G}}
\newcommand{\cS}{\mathcal{S}}
\newcommand{\cR}{\mathcal{R}}
\newcommand{\cA}{\mathcal{A}}

\newcommand{\cT}{\mathcal{T}}

\newcommand{\cC}{\mathcal{C}}

\newcommand{\cU}{\mathcal{U}}

\newcommand{\bpi}{\bm{\pi}}
\newcommand{\agr}{\operatorname{agr}}
\newcommand{\agrparam}[3]{\operatorname{agr}_{#1}\left(#2, #3\right)}
\newcommand{\bx}{\mathbf{x}}
\newcommand{\1}{\mathds{1}}
\renewcommand{\leq}{\leqslant}
\renewcommand{\geq}{\geqslant}
\renewcommand{\le}{\leqslant}
\renewcommand{\ge}{\geqslant}
\newcommand{\by}{\mathbf{y}}
\renewcommand{\mid}{\;:\;}

\newcommand{\PP}{\mathbb{P}}

\newcommand{\EE}{\mathbb{E}}
\newcommand{\N}{\mathbb{N}}
\newcommand{\RR}{\mathbb{R}}

\newcommand{\cAcodes}{\mathcal{C}}

\newcommand{\Stab}{\operatorname{Stab}}

\DeclareMathOperator{\poly}{poly}

\newtheorem{theorem}{Theorem}
\newtheorem{proposition}[theorem]{Proposition}
\newtheorem{lemma}[theorem]{Lemma}
\newtheorem{corollary}[theorem]{Corollary}
\newtheorem{definition}{Definition}
\newtheorem{claim}[theorem]{Claim}
\newtheorem{assumption}{Assumption}

\newtheorem{problem}{Probem}

\title{Forbidding just one intersection for short integer sequences}
\author{Elizaveta Iarovikova \quad Fedor Noskov \quad  Georgy Sokolov \quad Nikolai Terekhov}

\begin{document}

\maketitle

\begin{abstract}
	In this paper, we study the famous Erd\H{o}s--S\'os forbidden intersection problem for words over an alphabet of size $m$: what is the maximal size of a subfamily $\mathcal{F}$ of $[m]^n$ that does not contain two vectors $x, y$ coinciding on exactly $t - 1$ coordinates? We answer this question provided $m \ge \operatorname{poly}(t)$ and $n \ge \operatorname{poly}(t)$ for some polynomial function $\operatorname{poly}(\cdot)$ of $t$, greatly extending the recent result of Keevash, Lifshitz, Long and Minzer. Our proof combines some of the recently developed methods in extremal combinatorics, including the spread approximation technique of Kupavskii and Zakharov and the hypercontractivity approach developed in a series of works by Keevash, Keller, Lifshitz, Long, Marcus and Minzer.
\end{abstract}

\tableofcontents

\newpage

\section{Introduction}

Let $\binom{[n]}{k}$ be a family of $k$-subsets of $[n] = \{1, \ldots, n\}$. In 1971, Erd\H{o}s and S\'os (see~\cite{erdos1975problems}) considered the following problem.

\begin{problem}
\label{problem: standard erdos-sos}
Fix $n \ge k \ge t \ge 1$. What is the maximal size of a family $\ff \subset \binom{[n]}{k}$ such that for any $F_1, F_2 \in \ff$ we have $|F_1 \cap F_2| \neq t - 1$?
\end{problem}

Problem~\ref{problem: standard erdos-sos} is a harder version of the following question raised by Erd\H{o}s, Ko and Rado \cite{erdos1961intersection} ten years earlier.

\begin{problem}
\label{problem: erdos-ko-rado}
Fix $n \ge k \ge t \ge 1$. What is the maximal size of a family $\ff \subset \binom{[n]}{k}$ such that for any $F_1, F_2 \in \ff$ we have $|F_1 \cap F_2| \ge t$?
\end{problem}

Erd\H{o}s, Ko and Rado settled it for $n \ge n_0(k)$ in the famous Erd\H{o}s--Ko--Rado theorem~\cite{erdos1961intersection}. While Problem~\ref{problem: erdos-ko-rado} was completely solved by Ahlswede and Khachatrian~\cite{ahlswede1997complete} after several improvements due to Frankl~\cite{frankl1978erdos}, Wilson~\cite{wilson1984exact} and Frankl and F\"{u}redi~\cite{frankl1991beyond}, Problem~\ref{problem: standard erdos-sos} remains open. Erd\H{o}s and S\'os conjectured\footnote{The original conjecture of Erd\H os and S\'os is for families avoiding intersection size $1$, the general case is due to Erd\H os.} that any extremal solution of Problem~\ref{problem: standard erdos-sos} has the following form:
\begin{align*}
	\ff_{T} = \left \{ F \in \binom{[n]}{k} \mid T \subset F \right \}
\end{align*}
for some set $T$ of size $t$, provided $k \ge 2t$ and $n \ge n_0(k)$.\footnote{Note that $\ff_{T}$ is the unique up to an isomorphism extremal example for Problem~\ref{problem: erdos-ko-rado} if $n \ge (t + 1)(k - t + 1)$~\cite{frankl1978erdos, wilson1984exact}.} The conjecture was verified by Frankl and F\"uredi in the seminal paper~\cite{frankl_forbidding_1985}. However, the other regimes of parameters, when $t$ and $k$ grow with $n$, are also important and found numerous applications in discrete geometry~\cite{frankl1990partition, kahn1993counterexample}, communication complexity~\cite{sgall1999bounds} and quantum computing~\cite{buhrman1998quantum}. Great progress on the problem was made by Keller and Lifshitz who solved Problem~\ref{problem: standard erdos-sos} when $n$ is linear in $k$~\cite{keller2021junta}, i.e. $n \ge f(t) k$ for some function $f(\cdot)$. Later, Ellis, Keller and Lifshitz~\cite{ellis2024stability} extended the previous result to the regime when $k \le (1/2 - \varepsilon) n$ for any small enough $\varepsilon$ provided $n \ge n_0(t, \varepsilon)$. Recently, Kupavskii and Zakharov~\cite{kupavskii_spread_2022} managed to solve Problem~\ref{problem: standard erdos-sos} under the polynomial dependencies between $n, k, t$, e.g. when $n = k^\alpha$, $t = k^{\beta}$ with $0 < \beta < 1/2$ and $\alpha > 1+ 2 \beta$, provided $k \ge k_0(\alpha, \beta)$. Partial results were obtained by Cherkashin~\cite{cherkashin2024set} and Keevash, Mubayi and Wilson~\cite{keevash2006set}. It is also worth mentioning the celebrated Frankl and Wilson result~\cite{frankl1981intersection} that bounds the size of $\ff$ in Problem~\ref{problem: standard erdos-sos} by $\binom{n}{k - t}$ provided $k \ge 2t - 1$ and $k - t + 1$ is a prime power.

In recent decades, a rich and active direction of research has emerged: extending results from extremal set theory to families of various mathematical objects. The aforementioned Erd\H{o}s--Ko--Rado theorem~\cite{erdos1961intersection} has received numerous extensions to families of permutations~\cite{frankl1977maximum, cameron2003intersecting, larose2004stable, ellis2011intersecting,kupavskii_spread_2022,keller_t-intersecting_2024,kupavskii2024almost}, finite vector spaces~\cite{frankl1986erdos}, integer sequences~\cite{ahlswede1998diametric, frankl_erdos-ko-rado_1999}, matchings~\cite{kamat2013erdHos, fuentes2025intersecting}, independent sets of a graph~\cite{borg2009erdHos, borg2014erdos, hurlbert2011erdHos, woodroofe2011erdHos, frankl2024holroyd}(see also the survey~\cite{hurlbert2025survey}) and many other objects~\cite{chvatal2006intersecting,frankl2025intersecting,saengrungkongka2025t,iarovikova2025complete,kupavskii2023erd,kupavskii2023intersection}(see also the survey~\cite{ellis2022intersection}). 

In contrast, there are much fewer works on generalization of Problem~\ref{problem: standard erdos-sos} to other domains, mostly due to its hardness even in the $\binom{[n]}{k}$ case. There are some results for permutations~\cite{ellis2014forbidding, ellis2022approximation, keevash2017frankl, keller2025forbidden}, vector spaces~\cite{ellis2023forbidden} and integer sequences~\cite{keevash2023forbidden}. Significant progress in this direction was made by Kupavskii and Zakharov~\cite{kupavskii_spread_2022}. In particular, they present results for subfamilies of $[m]^n, \binom{[n]}{k}^w$, permutations and $(n, k, w)$-designs. However, assumptions of their theory turn out to be quite restrictive in our setup, which we discuss next. 

In this paper, we focus on Problem~\ref{problem: standard erdos-sos} adapted to subfamilies of $[m]^n$ for some integers $m, n$. For two vectors $x, y \in [m]^n$, we define their agreement $\agr(x, y)$ by
\begin{align*}
	\agr(x, y) = |\{i \in [n] \mid x_i = y_i\}|.
\end{align*}
Then, the question we study in this paper is as follows.
\begin{problem}
\label{problem: erdos-sos for codes}
Let $m, n, t$ be integers such that $m \ge 2$ and $n \ge t \ge 1$. What is the maximal size of a subfamily $\ff \subset [m]^n$ such that for any $x, y \in \ff$ we have $\agr(x, y) \neq t - 1$?
\end{problem}

We will refer to such families $\ff$ as \textit{$(t-1)$-avoiding}.

Note that for $t = 1$ this problem was resolved independently by Frankl and Tokushige~\cite{EKR_for_codes} and Ahlswede and Khachatrian~\cite{ahlswede1998diametric}. A more general result was obtained by Keevash and coauthors~\cite{keevash2023forbidden}.

\begin{theorem}[Theorem 1.1 of~\cite{keevash2023forbidden}]
\label{theorem: Keevash EKR for codes}
	Define families
	\begin{align*}
		\cS_{t,r}[m]^n = \{x \in [m]^n \mid |\{j \in [t + 2r] \mid x_j = 1\}| \ge t + r\}.
	\end{align*}
	Then, for any $t \in \N$ there is $n_0$ such that if $\ff \subset [m]^n$ is a $(t - 1)$-avoiding family with $m \ge 3$ and $n \ge n_0$, then $|\ff| \le \max_{r \ge 0} |\cS_{t,r}[m]^n|$ with equality only if $\ff$ is isomorphic to $\cS_{t,r}[m]^n$ for some $r \ge 0$.
\end{theorem}

The function $n_0(t)$ depends on $t$ super-exponentially in their proof. 

Observe that when $m > t$ the unique maximum in $\max_{r}|\cS_{t,r}[m]^n|$ is attained for $r = 0$. Thus, if $m > t$, then the only extremal families are $t$-stars $\{x \in [m]^n \mid x_{T} = z\}$ for some $T \subset [n]$ of size $t$ and $z \in [m]^T$, provided $n \ge n_0(t)$.

Another result was obtained by Kupavskii and Zakharov~\cite{kupavskii_spread_2022}.

\begin{theorem}[Proposition 23 of~\cite{kupavskii_spread_2022}]
\label{theorem: KZ erdos-sos for m-n}
Suppose that integers $m, n, t$ satisfy $$m \ge \max\{2^{15} \ln n, 2^{17} q \ln q\} \text{ and }n \ge 2 q$$ for $q = 800t^2 \ln m$. If $\ff \subset [m]^n$ is a $(t - 1)$-avoiding family and $|\ff| \ge 2/3 \cdot m^{n - t}$, then there exists a $t$-star $\cS_{(T, z)} = \{x \in [m]^n \mid x_{T} = z\}$ for some $T \subset [n]$ of size $t$ and $z \in [m]^T$ such that $|\ff \setminus \cS_{(T, z)}| \le m^{-4t}|\cS_{(T, z)}|$.
\end{theorem}

Despite the restrictions $m \ge 2^{15} \ln n$ and $n \ge 1600 t^2 \ln m$, their finding suggests that the result of Theorem~\ref{theorem: Keevash EKR for codes} can be potentially obtained when $m, n$ depend polynomially on $t$. Thus, it is natural to conjecture that the conclusion of Theorem~\ref{theorem: Keevash EKR for codes} holds when $m \ge \poly(t)$ and $n \ge \poly(t)$, without logarithmic lower bounds on $m$ and $n$. Here $\poly(\cdot)$ hides some polynomial functions. Our main result confirms this conjecture.

\begin{theorem}
    \label{theorem: main}
    There exist \textbf{polynomials} $n_0(\cdot)$ and $m_0(\cdot)$ such that for all $t \geq 1, n \geq n_0(t), m \geq m_0(t)$ the following holds. If $\mathcal{F} \subset [m]^n$ is $(t - 1)$-avoiding, then $|\mathcal{F}| \le m^{n-t}$ and the equality holds if and only if $\ff = \cS_{(T, z)}$ for some $T \subset [n]$ of size $t$ and $z \in [m]^T$.
\end{theorem}

\noindent \textbf{\textit{Remark.}} One can choose $m_0(t) = O(t^{300})$. The polynomial $n_0(t)$ can be replaced by a function $\tilde{n}_0(t) = C t^3 \ln t$ for some absolute constant $C$. 

\vspace{0.2cm}

The proof combines ideas from Keevash and coauthors~\cite{keevash2023forbidden}, Kupavskii and Zakharov \cite{kupavskii_spread_2022}, some recent developments of Kupavskii~\cite{kupavskii2023intersection}, a paper of Keller and coauthors on hypercontractivity for some pseudorandom functions~\cite{keller_sharp_2023} and a paper of Kupavskii and the second author~\cite{kupavskii2025lineardependenciespolynomialfactors}.

Our proof technique can be adapted to Problem~\ref{problem: standard erdos-sos}, yielding the solution when $k \ge \poly(t)$ and $n \ge \poly(t) \cdot k$, but we postpone it to more general theory~\cite{kupavskii_noskov_followup}.

The paper is organized as follows. In Section~\ref{section: notation}, we introduce the necessary notation. In Section~\ref{section: preliminaries}, we present the main tools used in the proof.  In Section~\ref{section: cross-disagreement in homogeneous families}, we develop some theory necessary for the analysis of the case $m \ge \poly(t)$ and $n \ge \poly(t) \ln m$. In Section~\ref{section: measure boosting argument case}, we present the solution of Problem~\ref{problem: erdos-sos for codes} when $m \ge \poly(t)$ and $n \ge \poly(t) \ln m$. In Section~\ref{section: classical extremal set theory case}, we present the solution of Problem~\ref{problem: erdos-sos for codes} when $m \ge n^{\poly(t)}$ and $n \ge \poly(t)$. Finally, in Section~\ref{section: proof of the main theorem}, we combine the results of Sections~\ref{section: measure boosting argument case} and~\ref{section: classical extremal set theory case} to prove Theorem~\ref{theorem: main}.

\section{Notation}
\label{section: notation}
We use standard notation $[n]$ for the set $\{1,\ldots, n\}$ and ${X\choose k}$ for the set of all $k$-element subsets of $X$. Given families $\ff$, $\mathcal{S}$ and a set $X$, we define
$$\ff(X) :=\{ A\setminus X  :\    A \in \ff, \,  X \subset A \}$$
$$\ff[X] :=\{ A  : \ A\in \ff, \,  X \subset A \}$$
$$\ff[ \mathcal{S} ] := \bigcup\limits_{X \in \mathcal{S}}\ff[X].$$
$$\ff( \mathcal{S} ) := \bigcup\limits_{X \in \mathcal{S}}\ff(X).$$

We denote $\partial_t\ff$ the set of all $t$-shadows of $\ff$, i.e. all $t$-element subsets of sets from $\ff$. We also denote the collection of all shadows of size at most $t$ as $$\partial_{\le t}\ff = \bigcup_{0\le q \le t}\partial_q\ff.$$ Lastly, by $\ff\vee\{T\}$ we mean $\{F \cup T \mid F\in\ff\}$.

We say that two sets $F_1, F_2$ form a matching, if $F_1 \cap F_2 = \varnothing$. Two families of sets $\ff_1, \ff_2$ cross-contain a matching (or contain a cross-matching) if there are two sets $F_1 \in \ff_1$, $F_2 \in \ff_2$ such that $F_1 \cap F_2 = \varnothing$.

Depending on the context, we will identify the families of codes $\ff \subset [m]^n$ with families of sets from some ambient family $\cAcodes \subset \binom{I_1 \sqcup I_2 \ldots \sqcup I_n}{n}$, where $I_i$ consists of integers from $(i - 1)m + 1$ to $im$. We identify each code $x \in [m]^n$ with a set $X$ that intersects each $I_i$ by the element $(i - 1)m + x_i$ only. Then $\cAcodes$ is a family of all such sets $X$:
\begin{align}
\label{eq: codes set family definition}
    \cAcodes = \left \{X \in \binom{[mn]}{n} \mid \text{for all }i \in [n] \text{ we have }|X \cap I_i| = 1 \right \}.
\end{align}

Given $Z \subset [n]$ and $x \in [m]^T$, we call $x$, or, if we want to emphasize the set of coordinates, $(Z, x)$, a \textit{partial code} or a \textit{restriction}. Partial codes are shadows of elements of $[m]^n$. We identify them with shadows of elements of $\cAcodes$ in the following way: a restriction $(Z, x)$ corresponds to $$T = \{i(z-1) + x_z\mid z\in Z\} \in \partial_{|Z|}\cAcodes.$$ We write $$T = (Z, x); \quad \ff[Z \to x] = \ff[T]; \quad \ff(Z \to x) = \ff(T).$$

If $|Z| = 1,$ i.e. $Z = \{i\}$ for some $i \in [n]$, we may write $\ff[i \to x]$ and $\ff(i \to x)$ instead of $\ff[\{i\} \to x]$ and $\ff(\{i\} \to x)$.

Given an integer $t$, a set $T$ of size $t$, and a code $z \in [m]^T$, we refer to a family $\{x \in [m]^n \mid x_T = z\}$ as $t$-star. We call subfamilies of $t$-stars incomplete $t$-stars.

Sometimes we consider a restriction of a code or a family of codes on a particular set of coordinates. Given $Z \subset Z_1 \subset [n]$ and $y \in [m]^{Z_1}$, by $y|_Z$ we mean $y' \in [m]^Z$ such that $y_z = y'_z$ for all $z \in Z$. For a family $\ff$ we denote its restriction $\ff|_Z = \{y|_Z \mid y \in \ff\}$. We are also interested in the agreement of two partial codes on a subset of the set of their common coordinates. Specifically, given sets $Z_1, Z_2 \subset [n]$, $Z \subset Z_1 \cap Z_2 $ and partial codes $x_1 \in [m]^{Z_1}$, $x_1 \in [m]^{Z_2}$, we denote $$\agr_Z(x_1, x_2) = \agr(x_1|_Z, x_2|_Z).$$

Finally, we discuss the notation concerning measures and their restrictions. We consider probability measures on $[m]^n$ or $[m]^Z, Z\subset [n]$. Two common examples are the uniform measure $\mu$ and the class of product measures, i.e., the measures $\nu$ that could be represented as $\nu(x) = \prod_{i \in Z}\nu_i(x_i)$, where $\nu_i$ are some probability measures on $[n]$. Note that a uniform measure is also a product measure although we mention it separately. Given a partial code $(Z, x)$ and a measure $\nu$ on $[m]^n$, the restricted measure $\nu_{Z \to x}$ on $[m]^n(Z \to x)$ is defined as $$\nu_{Z \to x}(y) = \frac{\nu((x, y))}{\nu([m]^n[Z \to x])},$$ where $(x, y)$ is a code that coincides with $x$ on $Z$ and with $y$ on $[n]\setminus Z$. If $Z = \{i\}$, we sometimes write $\nu_{i \to x}$ instead of $\nu_{\{i\} \to x}$.

For a product measure $\nu = \bigotimes_{i=1}^n\nu_i$, we define its restriction on the set of coordinates $Z$ as $$\nu_{Z} = \prod_{i \in Z}\nu_i.$$ Note that, for a product measure, $\nu_{Z \to x}$ does not depend on $x$, and we have $\nu_{Z \to x} = \nu_{[n]\setminus Z}$. 



\section{Preliminaries}
\label{section: preliminaries}

        We treat the regimes $\poly(t) \le m \le e^{Cn/ t^2}$ and $m \ge n^{Ct}$ separately, and each regime requires a completely different set of tools.  Note that for our assumptions on $n$ these regimes overlap. 
        
        We treat the case $m \le e^{Cn/t^2}$ using a measure boosting argument, which dates back to Dinur and Friedgut~\cite{dinur_intersecting_2009}. We illustrate it on subfamilies of sets from some ambient family $\cA$ (which is $\cC$ in our case) endowed with the uniform measure $\mu(\cdot)$.  The argument is based on the following informal observation: \textit{given two pseudorandom families $\ff_1, \ff_2$, under mild assumptions on their measure $\mu(\cdot)$ in some ambient family $\cA$, one can obtain families $\widetilde{\ff}_1, \widetilde{\ff}_2$ in (probably, different) ambient family $\widetilde{\cA}$ such that 
        \begin{enumerate}
            \item they have a constant volume w.r.t. some other measure $\widetilde{\mu}(\cdot)$ on $\widetilde{\cA}$;
            \item if $\widetilde{\ff}_1, \widetilde{\ff}_2$ cross-contain a matching, then $\ff_1, \ff_2$ do.
        \end{enumerate}}
        The final trick of this argument is that one can easily prove the existence of cross-matching in families of large measure, leading to the conclusion: \textit{pseudorandom families cross-contain a matching}. It turns out that, in many Tur\'an-type problems, extremal families admit decomposition into relatively simple pseudorandom subfamilies, and the above reasoning allows to relax conditions on problem parameters.

        The efficiency of this idea significantly varies with the notion of pseudorandomness and the range family $\widetilde{\cA}$ used. The original definition of pseudorandomness due to Dinur and Friedgut called \textit{uncapturability} plays the key role in the celebrated junta method~\cite{keller2021junta, ellis2023forbidden, keevash2023forbidden, lifshitz2020set}. Later, Alweiss and coathours~\cite{alweiss2021improved} made a breakthrough result in the Erd\H{o}s sunflower problem using another notion of pseudorandomness called \textit{spreadness}. Using their results together with delicate combinatorial arguments, Kupavskii and Zakharov~\cite{kupavskii_spread_2022} developed the spread approximation technique for which they introduced a new concept of pseudorandomness named $\tau$-\textit{homogeneity}. They apply their method to obtain solutions of the Deza--Frankl problem on $t$-intersecting permutations and the Erd\H{o}s--S\'os forbidden intersection problem under polynomial dependencies between parameters. The technique was further improved in subsequent works of Kupavskii and coathours~\cite{kupavskii2023erd,kupavskii2023intersection,kupavskii2024almost,kupavskii2025lineardependenciespolynomialfactors, frankl2025hajnal, iarovikova2025complete} yielding a large variety of results in extremal set theory. In particular, in this paper we enhance the spread approximation technique for the families of codes. At the same time, Keevash and coathours~\cite{keevash2021global,keevash2023forbidden} presented the notion of \textit{global} families, and using it they reached large progress in Tur\'an-type problems on forbidden large sparse hypergraphs. Then, independently of Kupavskii and Zakharov, they modified their definition of global families (so it coincides now with the $\tau$-homogeneity), and generalized it for functions on a product space~\cite{keller_sharp_2023}. The main result of their findings is a new hypercontractivity inequality for global functions, which we use below; see Theorem~\ref{theorem: hypercontractivity}. In follow-up papers, they show that their approach has a lot of consequences in different fields of discrete mathematics~\cite{gill2024initiating, lifshitz2024completing, keller2024improved, filmus2024hypercontractivity}. 

		In this paper, we modify the measure boosting argument of~\cite{keevash2023forbidden} for families of codes, using the ideas from~\cite{kupavskii2025lineardependenciespolynomialfactors,keller_sharp_2023}. Starting from families $\ff_1, \ff_2 \subset \cA = [m]^n$ and the uniform measure $\mu(\cdot)$, we obtain families $\widetilde{\ff}_1, \widetilde{\ff}_2 \subset \widetilde{\cA}$ and a measure $\widetilde{\mu}(\cdot)$, iteratively applying an operation which the authors of~\cite{keevash2023forbidden} call \textit{gluing}.

		\subsection{Random gluings}

		We start with some definitions.

		\begin{definition}[Definition 6.5~\cite{keevash2023forbidden}]
		\label{definition: gluing}
		We refer to a map $\pi: [m_1] \to [m_2]$ as a gluing in the sense, that it merges elements of $\pi^{-1}(y), y \in [m_2]$, into one element $y$. A gluing $\pi$ is called \textbf{$b$-balanced}, if
		\begin{align*}
			|\pi^{-1}(y)| \le \frac{b m_1}{m_2}.
		\end{align*}
		The family of $b$-balanced gluings from $[m_1]$ to $[m_2]$ is denoted by $\Pi_{m_1, m_2, b}$. If $m_2 \, | \, m_1$ and $b = 1$, we omit subscript $b$.

		A $b$-balanced gluing $\pi: [m_1]^n \to [m_2]^n$ is defined as a tensor product of $b$-balanced gluings $\pi_1, \ldots, \pi_n$, i.e. $\pi(x_1, \ldots, x_n) = (\pi_1(x_1), \ldots, \pi_n(x_n))$. The set of $b$-balanced gluings from $[m_1]^n$ to $[m_2]^n$ by $\Pi_{m_1, m_2, b}^{\otimes n}$. As before, the subscript $b$ is omitted if $b = 1$ and $m_2 \, | \, m_1$.
		\end{definition}

		Then, for a family $\ff \subset [m_1]^n$ and a gluing $\pi$, we define
		\begin{align*}
			\ff^\pi = \{\pi(x) \mid x \in \ff\}.
		\end{align*}
        If $\ff \subset [m]^S$ for some $S \subset [n]$, then, slightly abusing notation, we denote $\ff^{\pi|_S}$ by $\ff^\pi$ as well.
        
		Note that if $\ff_1^\pi, \ff_2^\pi$ cross-contain a disagreement, then $\ff_1, \ff_2$ do.
		Similarly, for a product measure $\nu = \bigotimes_{i = 1}^n \nu_i$ on $[m_1]^n$ and a gluing $\pi = \bigotimes_{i = 1}^n \pi_i \in \Pi_{m_1, m_2, b}^{\otimes n}$, we introduce a measure $\nu^{\pi}(\cdot)$ on $[m_2]^n$, defined as follows:
		\begin{align*}
			\nu^\pi(\{y\}) = \prod_{i = 1}^n \nu(\pi^{-1}_i(y_i)).
		\end{align*}
		Let us comment now on the $b$-balanced property of a random gluing. In what follows, it will be important to ensure that if a measure $\nu$ is similar to a uniform measure in some sense, then the measure $\nu^{\pi}(\cdot)$ also is. We control this similarity introducing the notion of $b$-balanced measures.

		\begin{definition}[Definition 6.7~\cite{keevash2023forbidden}]
		\label{definition: b-balanced measure}
		A product measure $\nu= \bigotimes_{i = 1}^n \nu_i$ on $[m]^n$ is called $b$-balanced if 
		\begin{align*}
			\nu_i(\{x\}) \le \frac{b}{m}
		\end{align*}
		for any $x \in [m]$ and $i \in [n]$.
		\end{definition}

		One can easily verify the following claim.
		
		\begin{claim}
		\label{claim: balanced consistency}
			If a product measure $\nu(\cdot)$ on $[m_1]^n$ is $b_1$-balanced and a gluing $\pi: [m_1]^n \to [m_2]^n$ is $b_2$-balanced, then $\nu^{\pi}$ is $b_1 b_2$-balanced.
		\end{claim}

        \subsubsection{Gluings and the measure boosting argument}

		Following~\cite{keevash2023forbidden}, for the measure boosting argument, we will construct $\widetilde{\mu}(\cdot)$ as a measure $\mu^{\widetilde{\pi}}(\cdot)$, where $\mu(\cdot)$  stands for the uniform measure on some product space $[m]^{n'}$, $n' \le n$, and $\widetilde{\pi}$ is some gluing. Let us first point out that any gluing does not decrease a measure of a family $\ff \subset [m]^n$.
		\begin{claim}[Claim 6.8~\cite{keevash2023forbidden}]
		\label{claim: measure consistency after a gluing}
		For any $\ff \subset [m_1]^n$, a gluing $\pi: [m_1]^n \to [m_2]^n$ and product measure $\nu$ on $[m_1]^n$, we have $\nu^\pi(\ff^\pi) \ge \nu(\ff)$.
		\end{claim}

		Clearly, Claim~\ref{claim: measure consistency after a gluing} is insufficient for upgrading the measure of a family $\ff \subset [m]^n$. The following lemma shows how one can enlarge the measure of a family $\ff$ choosing a gluing $\pi$ via some averaging argument. The statement of the lemma requires some basic knowledge concerning Markov chains; here we provide some definitions.

		For a measure $\nu$ on $[m]^n$, define the scalar product $\langle \cdot, \cdot \rangle_{\nu}$ on the space $L_2([m]^n, \nu)$ as
		\begin{align*}
			\langle f, g \rangle_{\nu} = \EE_{\bx \sim \nu} f(\bx) g(\bx).
		\end{align*}

		Then, let $T$ be a Markov chain with the stationary measure $\nu$. We define a linear operator $T: L_2([m]^n, \nu) \to L_2([m]^n, \nu)$ as follows. Let $\by$ be a step of the Markov chain $T$ starting from a point $x$. Then, we have
		\begin{align*}
			(T f)(x) = \EE f(\by).
		\end{align*}
		Finally, for a Markov chain $T$ with the stationary measure $\nu$ define $\lambda_*(T)$ as follows
		\begin{align*}
			(1 - \lambda_*(T))^2 = \sup \{ \EE_{\bx \sim \nu} (T f(\bx))^2 \mid \EE_{\bx \sim \nu} f(\bx) = 0, \; \EE_{\bx \sim f} f^2(\bx) = 1 \}.
		\end{align*}

		Then, we are ready to state the lemma.

		\begin{lemma}
		\label{lemma: random gluing measure upgrade}
		Let $\nu(\cdot)$ be a product measure on $[m_1]^n$. Fix $m_2, b$ such that $m_2 \le m_1 \le b m_2$. Then, there exists a product Markov chain $T = \otimes_{i = 1}^n T_i$ with the stationary measure $\nu(\cdot)$, such that
		\begin{align}
		\label{eq: Markov chain lower bound}
			\EE_{\bpi \sim U(\Pi_{m_1, m_2, b}^{\otimes  n})}\nu^{\bpi}(\ff) \ge \frac{\nu^2(\ff)}{\langle \1_{\ff}, T (\1_{\ff}) \rangle_{\nu} },
		\end{align}
		where $U(\Pi_{m_1, m_2, b}^{\otimes n})$ stands for the uniform measure on $\Pi_{m_1, m_2, b}^{\otimes n}$, $\1_{\ff}$ stands for the characteristic function of $\ff$.
		Moreover, we have $\lambda_*(T) \ge 1/6$.
		\end{lemma}

		For the proof, we refer the reader to that of Lemma 6.9~\cite{keevash2023forbidden}. 

		It will be convenient to upper bound the scalar product in the right-hand side of~\eqref{eq: Markov chain lower bound}, replacing the Markov chain operator $T$ with some operator defined as follows. Given a vector $x \in [m]^n$ and a product measure $\nu = \bigotimes_{i = 1}^n \nu_i$, define a distribution $N_{\rho}(x) = \otimes_{i = 1}^n N_{\rho}(x_i)$, where $N_\rho(x_i)$ is a distribution of the following random variable $\by_i$:
		\begin{align*}
			\by_i = \begin{cases}
				x_i, & \text{ with probability } \rho, \\
				\text{sample from } \nu_i(\cdot), & \text{ with probability } 1 - \rho.
			\end{cases}
		\end{align*}
		Then, we define an operator $T_{\rho}: L_2([m]^n, \nu) \to L_2([m]^n, \nu)$ by the averaging a function $f$ over $\by$:
		\begin{align*}
			(T_\rho f)(x) = \EE_{\by \sim N_\rho(x)} f(\by).
		\end{align*}
        We also introduce a quantity that is closely related to the noise operator. Namely, we define the \textit{stability} of a function $f$ or a family $\ff$ as
        \begin{align*}
			\Stab_{\rho}(f) = \langle f, T_\rho f \rangle_{\nu} \quad \text{and} \quad \Stab_{\rho}(\ff) = \langle \1_{\ff}, T_\rho \1_{\ff} \rangle_{\nu}.
		\end{align*}
        
		Then, we upper bound $\langle f, T f \rangle_{\nu}$ by the following lemma.

		\begin{lemma}[Lemma 5.13\cite{keevash2023forbidden}]
		\label{lemma: stability upper bound}
		Suppose that a product Markov chain $T$ has $\lambda_*(T) \ge 1/6$. Then, we have
		\begin{align*}
			\langle f, T f \rangle_{\nu} \le \Stab_{5/6}(f).
		\end{align*}
		\end{lemma}

		Lemma~\ref{lemma: stability upper bound} and Lemma~\ref{lemma: random gluing measure upgrade} yield the following corollary.
		\begin{corollary}
		\label{corollary: stab measure upgrade}
		Fix numbers $n, m_1, m_2, b$ such that $m_2 \le m_1 \le b m_2$. Then, for a function $f \in L_2([m_1]^n, \nu)$, a family $\ff$ and a product measure $\nu$, there exists a gluing $\pi\in \Pi_{m_1, m_2, b}^{\otimes n}$ such that
		\begin{align*}
			\nu^{\pi} (\ff^{\pi}) \ge \frac{\nu^2(\ff)}{\Stab_{5/6}(\ff)}.
		\end{align*}
		\end{corollary}

		Due to Corollary~\ref{corollary: stab measure upgrade}, for the measure boosting argument, we need to control $\Stab_{5/6}(\ff)$. The following proposition ensures that it is enough to upper bound $\Stab_{\rho_0}(\ff)$ for sufficiently small $\rho_0$.

		\begin{proposition}[Lemma 5.11~\cite{keevash2023forbidden}]
		\label{proposition: smaller rho stability upper bound}
		For any function $f \in L_2([m]^n, \nu)$, we have
		\begin{align*}
			\Stab_{\rho}(f) \le \Vert f \Vert_2^{2(1 - 1/t)} \Stab_{\rho^t}(f)^{1/t}
		\end{align*}
		whenever $t = 2^d$ with $d \in \N$.
		\end{proposition}

		To control $\Stab_{\rho_0}(\ff)$, one need to guarantee some pseudorandomness of $\ff$. Authors of the original paper~\cite{keevash2023forbidden} used the notion of globalness due to Keevash et al.~\cite{keevash2021global}. Instead, we use the concept of globalness from~\cite{keller_sharp_2023}.

\begin{theorem}[Theorem 1.4 from~\cite{keller_sharp_2023}]
\label{theorem: hypercontractivity}
Say that a function $f$ on a product space $\prod_{i = 1}^n \Omega_i$ endowed with a product measure $\nu = \otimes_{i = 1}^n \nu_i$ is $r$-global if for any $S \subset [n]$ and any $x \in \prod_{i \in S} \Omega_i$, we have
\begin{align*}
	\Vert f_{S \to x} \Vert_2 \le r^{|S|} \Vert f \Vert_2,
\end{align*}
where $f_{S \to x} : \prod_{i \in [n] \setminus S} \Omega_i \to \RR$ stands for the function $f_{S \to x}(y) = f(x, y)$. 
Let $q \ge 2$. If $r \ge 1$ and $\rho \le \frac{\ln q}{32 r q}$, then $\Vert T_\rho f\Vert_q \le \Vert f \Vert_2$ for any $r$-global function $f$.
\end{theorem}

In the paper~\cite{keller_sharp_2023}, the proof of Theorem~\ref{theorem: hypercontractivity} is presented for the case when all $\Omega_i$ and $\nu_i$ coincide. However, it works for the product space since it based on the freezing argument and induction.

We say that a family $\ff$ is $\tau$-homogeneous if $\1_{\ff}$ is $\sqrt{\tau}$-global. This yields the following definition.

\begin{definition}
\label{definition: homgeneous families}
Say that $\ff \subset [m]^n$ is $\tau$-homogeneous w.r.t. to a product measure $\nu$ if for any $S$ and $x \in [m]^S$ we have
\begin{align*}
	\nu_{S \to x}(\ff(S \to x)) \le \tau^{|S|} \nu(\ff).
\end{align*} 
\end{definition}

We discuss $\tau$-homogeneous families in more details in the subsequent section. For now, let us present how to bound $\Stab_{\rho_0}(\ff)$ for $\tau$-homogeneous family $\ff$. By the H\"{o}lder inequality, we have
\begin{align*}
	\Stab_{\rho_0}(\ff) \le \Vert T_{\rho_0} \1_{\ff} \Vert_4 \Vert \1_{\ff} \Vert_{4/3}.
\end{align*}
If $\rho_0$ satisfies the conditions of Theorem~\ref{theorem: hypercontractivity} for $q = 4$, we have
\begin{align*}
	\Stab_{\rho_0}(\ff) \le \Vert \1_{\ff} \Vert_2 \Vert \1_{\ff} \Vert_{4/3} = \nu^{5/4}(\ff).
\end{align*}
It turns out, that this bound is sufficient for our purposes, see the proof of Lemma~\ref{lemma: boosting measure step}.

Given two homogeneous families $\ff_1, \ff_2$, suppose for a while that, iteratively using Corollary~\ref{corollary: stab measure upgrade}, we found families $\widetilde{\ff}_1, \widetilde{\ff}_2 \subset [m']^{n'}$ such that $\widetilde{\mu}(\widetilde{\ff}_1), \widetilde{\mu}(\widetilde{\ff}_2)$ is larger than some constant and $\widetilde{\mu}$ is an enough balanced product measure. Then, the following lemma allows us to find a cross-disagreement in families $\widetilde{\ff}_1, \widetilde{\ff}_2$.

\begin{lemma}[Lemma 5.9~\cite{keevash2023forbidden}]
\label{lemma: hoffman bound}
Let $\nu = \bigotimes_{i = 1}^n \nu_i$ be a product probability measure on $[m]^n$ such that $\nu_i(x) \le \lambda \le 1/2$ for all $i \in [n]$ and $x \in [m]$. Suppose that $\cG_1, \cG_2 \subset [m]^n$ are cross-intersecting with $\nu(\cG_i) = \alpha_i$ for $i = 1, 2$. Then
\begin{align*}
    \alpha_1 \alpha_2 \le \left ( \frac{\lambda}{1 - \lambda} \right )^2 (1 - \alpha_1) (1 - \alpha_2).
\end{align*}
\end{lemma}

There is an alternative measure boosting argument that performs well if one of the initial measures $\mu(\ff_1), \mu(\ff_2)$ is large. We present it as the following lemma. While its statement can not be found in paper~\cite{keevash2023forbidden}, the proof coincides with that of Theorem~2.3~\cite{keevash2023forbidden}. For the sake of completeness, we provide the proof in the appendix.

\begin{lemma}
    \label{lemma: unbalanced cross matching}
    Let $\mathcal{F}_1, \mathcal{F}_2 \subset [m]^{n}$ and let $\mu$ be uniform measure on $[m]^n$. If $\mu(\mathcal{F}_1) + \mu(\mathcal{F}_2)^{\log_{m}(2)} > 1$, then there exist $x_i \in \mathcal{F}_i$ such that $\agr(x_1, x_2) = 0$.
\end{lemma}

In the next section, we move on to the following questions: how to decompose a family $\ff$ into homogeneous families and how to reduce a question about avoiding agreement of size exactly $t - 1$ into a problem of finding a cross-disagreement in homogeneous families?
	
	\subsection{Homogeneous families}

	First, we discuss how to decompose a family $\ff$ into homogeneous families $\ff_i$ up to some small remainder. We start with the following observation.
	\begin{lemma}
	\label{lemma: homogenous restriction}
		Let $\mathcal{F} \subset [m]^n$. Let $\nu$ be a product measure on $[m]^n$ and fix $\tau > 1$. If $Z \subset [n]$ is maximal such that for some $r \in [m]^{Z}$ we have $\nu_{Z \to r}(\mathcal{F}(Z \to r)) \ge \tau^{|Z|} \nu(\mathcal{F})$, then $\mathcal{F}(Z \to r)$ is $\tau$-homogenous with respect to $\nu_{Z \to r}$. Note that if $Z = [n]$, then $\nu_{Z \to r}(\ff(Z \to r)) = \1\{r \in \ff\}$.
	\end{lemma}

	Using the above lemma iteratively, one can obtain the decomposition which is called the spread approximation of $\ff$.
	\begin{lemma} [Lemma 52 from~\cite{kupavskii2025lineardependenciespolynomialfactors}]
		\label{lemma: spread approximation}
		Fix an integer $q$ and a real number $\tau$. Let $\mathcal{F} \subset [m]^n$ and let $\mu$ be the uniform measure on $[m]^n$. There exists a system of pairs $\{(Z_i, x_i)| i = 1, \ldots, l\}$ such that for all $i$ we have $Z_i \subset [n], |Z_i| \leq q$, $x_i \in [m]^{Z_i}$ and a partition of $\mathcal{F}$ into families $\mathcal{F}_i \subset \mathcal{F}[Z_i \rightarrow x_i]$ and $\mathcal{R}$ such that
		$$\mathcal{F} = \mathcal{R} \sqcup \bigsqcup_{i = 1}^l \mathcal{F}_i,$$
		$\mu(\mathcal{R}) \leq \tau^{-q}$ and for all $i$ we have $\mu_{Z_i \rightarrow x_i}(\mathcal{F}_i(Z_i \rightarrow x_i)) \geq \tau^{-q}$ and $\mathcal{F}_i(Z_i \rightarrow x_i)$ is $\tau$-homogenous with respect to $\mu_{Z_i \rightarrow x_i}$.
	\end{lemma}
	In the paper~\cite{kupavskii2025lineardependenciespolynomialfactors}, the authors require $\mu(\cR) \le 32 \tau^{-q}$, but their proof works with the tighter bound suggested in the above lemma. For the sake of completeness, let us present the proof of Lemma~\ref{lemma: spread approximation}.

	\begin{proof}[Proof of Lemma~\ref{lemma: spread approximation}]
		We will obtain the decomposition into families $\cR, \ff_1, \ldots, \ff_l$ at the end of the following procedure. Define $\cG_0 = \ff$. In the step $i$, $i = 1, 2, \ldots$, find a maximal set $Z_i \subset [n]$ such that there exists $x_i \in [m]^{Z_i}$ with the following properties:
		\begin{align*}
			\mu_{Z_i \to x_i}(\cG_{i - 1}(Z_i \to x_i)) \ge \tau^{|Z_i|} \mu(\cG_{i - 1}).
		\end{align*}
		If such $Z_i$ does not exist, set $l = i$, $Z_i = \varnothing$, $\ff_l = \cG_{l - 1}$. Otherwise, if $|Z_i| > q$ or $\mu_{Z_i \to x_i}(\cG_{i - 1}(Z_i \to x_i)) \le \tau^{-q}$, then stop, put $l = i - 1$ and set $\cR = \cG_{i - 1}$. Finally, if $|Z_i| \le q$ and $\mu_{Z_i \to x_i}(\cG_{i - 1}(Z_i \to x_i)) \ge \tau^{-q}$, then set $\ff_i = \cG_{i - 1}[Z_i \to x_i]$, $\cG_i = \cG_{i - 1} \setminus \ff_{i}$ and move to the step $i + 1$. 

		In the end of procedure, we have either $\mu(\cR) \le \tau^{- |Z_{l + 1}|} \mu_{Z_{l + 1} \to r_{l + 1}}(\cG_{l}(Z_{l + 1} \to x_{l + 1}))$ for some $|Z_{l + 1}| > q$ or $\mu(\cR) \le \tau^{-q}$. In both cases, we have $\mu(\cR) \le \tau^{-q}$. By Lemma~\ref{lemma: homogenous restriction}, families $\ff_i$ are $\tau$-homogeneous, and the proof is complete.
	\end{proof}

	The core idea of our proof is to show that an extremal family $\ff \subset [m]^n$ avoiding an agreement of size exactly $t - 1$ is $t$-agreeing. At the first stage, it is enough to show that restrictions $(Z_i, x_i), i = 1, \ldots, l$, of the spread approximation are $t$-agreeing, so $\ff$ is $t$-agreeing up to small remainder. Arguing indirectly, suppose that there are two restrictions $(Z_1, x_1)$ and $(Z_2, x_2)$ in $\{(Z_i, x_i)\}_{i = 1}^l$ such that $\agr_{Z_1 \cap Z_2}(x_1, x_2) < t$. Then, we will find a set $H \subset [n]$ of size $t - 1 - \agr_{Z_1 \cap Z_2}(x_1, x_2)$ and a vector $y \in [m]^H$ such that  such that $\agr_{(Z_1 \cap Z_2) \cup H}((x_1, y), (x_2, y)) = t - 1$, and some families $\ff_1' \subset \ff(Z_1 \cup H \to (x_1, y))$, $\ff_2 \subset \ff(Z_2 \cup H \to (x_2, y))$ are $\tau'$-homogeneous and have a moderately large measure, see Lemma~\ref{lemma: big common restriction} below. Then, the measure boosting argument of the previous section guarantees that they cross-contain a disagreement, a contradiction.

	Thus, we require some tools to work with restrictions of $\tau$-homogeneous families. In particular, we need the following lemma.

    \begin{lemma}
		\label{lemma: homogenous avoiding}
		Let $\mathcal{F} \subset [m]^n$ be $\tau$-homogenous with respect to a $b$-balanced product measure $\nu$. Let $R \subset [n]$. Let $X \subset [m]^R$ be a set such that $X = \prod_{i \in R}X_i$ for some $X_i \subset [m]$, $\sum_{i \in R}|X_i| \tau b < m$. Define $$\mathcal{F}' = \{y \in \mathcal{F}| \forall x \in X \, \agr(x, y|_{R}) = 0\}.$$ Then, $\nu(\mathcal{F}') \geq (1 - \frac{\sum_{i\in R}|X_i|\tau b}{m})\nu(\mathcal{F})$. Moreover, $\mathcal{F}'$ is $\tau'$-homogenous for 
		\begin{align*}
			\tau' = \tau \cdot \left(1 - \frac{\sum_{i\in R}|X_i|\tau b}{m}\right)^{-1}.
		\end{align*}
	\end{lemma}

	\begin{proof}
		For any $i\in R, x \in X_i$ we have $$\nu(\mathcal{F}[i \rightarrow x]) = \nu_{i}(x) \nu_{i \to x}(\mathcal{F}(i \rightarrow x)) \leq \nu_{i}(x) \tau\nu(\mathcal{F}),$$
		where the last inequality holds by the $\tau$-homogeneity of $\ff$.
		Since $\nu$ is $b$-balanced, we have $\nu_{i}(x)\leq \frac{b}{m}$. It implies
		\begin{align*}
			\nu(\mathcal{F}') & = \nu(\mathcal{F} \setminus \bigcup_{i\in R}\bigcup_{x\in X_i}\mathcal{F}[i \rightarrow x])  \geq \nu(\mathcal{F}) - \sum_{i\in R}\sum_{x \in X_i}\nu(\mathcal{F}[i \rightarrow x]) \\
			& \geq  \left(1 - \frac{\sum_{i\in R}|X_i|\tau b}{m}\right) \cdot \nu(\mathcal{F}),
		\end{align*}
		and the first statement of the lemma holds. Let us check the second statement on $\tau'$-homogeneity of $\ff'$. Consider arbitrary $Z \subset [n]$ and $r \in [m]^Z$. We have
		\begin{align*}
			\nu_{Z\rightarrow r}(\mathcal{F}'(Z \rightarrow r)) & \leq \nu_{Z\rightarrow r}(\mathcal{F}(Z \rightarrow r)) \leq \tau^{|Z|}\nu(\mathcal{F}) \leq \frac{\tau^{|Z|}}{\left(1 - \frac{\sum_{i\in R}|X_i|\tau b}{m}\right)}\nu(\mathcal{F}') \\
            & \leq \left(\frac{\tau}{1 - \frac{1}{m}\sum_{i\in R}|X_i|\tau b}\right)^{|Z|}\nu(\mathcal{F}')
		\end{align*}
		and, therefore, $\mathcal{F}'$ is $\tau'$-homogenous.
	\end{proof}

    To find certain restrictions, we widely use a simple averaging argument, which we state as a separate claim.
    
    \begin{claim}
    \label{claim: averaging argument}
    Let $\ff$ be a subset of $[m]^n$ and let $\nu(\cdot)$ be a product measure on $[m]^n$. Then, for any $H \subset [n]$ there exists $x \in [m]^H$ such that
    \begin{align*}
        \nu_{H \to x}(\ff(H \to x)) \geq \nu(\ff).
    \end{align*}
    \end{claim}
    \begin{proof}
        Consider a random vector $\bx$ distributed on $[m]^H$ according to the measure $\prod_{i \in S} \nu_i$. Then, we have
        \begin{align}
            \EE \nu_{H \to \bx}(\ff(H \to \bx)) = \sum_{x \in [m]^H} \prod_{i \in H} \nu_i(x_i) \cdot \frac{\nu(\ff[H \to x])}{\prod_{i \in H} \nu_i(x_i)} = \sum_{x \in [m]^H} \nu(\ff[H \to x]) = \nu(\ff),
			\label{eq: double counting -- expectation}
        \end{align}
        and the claim follows.
    \end{proof}

    Another helpful argument that is based on similar ideas is the following lemma, which is applied to a pair of $\tau$-homogenous families simultaneously and allows to find a common restriction such that both restricted families have large measure.
    \begin{lemma}
	\label{lemma: large restriction with large probavility}
	Let $\mathcal{F} \subset [m]^n$ be a $\tau$-homogenous family with respect to a product measure $\nu$. Let $H \subset [n]$ and let $\bx \sim \nu_{H}$ be a random element in $[m]^H$. Let $p \in (0, 1)$. If $\tau < \left(\frac{1+p}{2p}\right)^{\frac{1}{|H|}}$, then
	$$\PP \left (\nu_{H \rightarrow \bx}(\mathcal{F}(H \rightarrow \bx)  ) \geq \frac{1}{2} \nu(\mathcal{F}) \right ) > p.$$
\end{lemma}

\begin{proof}
	Let $q = \PP \left (\nu_{H \rightarrow \bx}(\mathcal{F}(H \rightarrow \bx)) \geq \frac{1}{2} \nu(\mathcal{F}) \right )$.
	We count $\EE \nu_{H \rightarrow \bx}(\mathcal{F}(H \rightarrow \bx))$ in two ways. 
	On the one hand, since $\nu_{H \rightarrow x}(\mathcal{F}(H \rightarrow x)) \leq \tau^{|H|}\nu(\mathcal{F})$ for all $x$, we have 
    \begin{align*}
        \EE \nu_{H \rightarrow \bx}(\mathcal{F}(H \rightarrow \bx)) & = \EE\nu_{H \rightarrow \bx}(\mathcal{F}(H \rightarrow \bx)) \cdot \1 \left \{\nu_{H \rightarrow \bx}(\mathcal{F}(H \rightarrow \bx))  \geq \frac{1}{2} \nu(\mathcal{F}) \right \} \\
        & \quad + \EE \nu_{H \rightarrow \bx}(\mathcal{F}(H \rightarrow \bx)) \1 \left \{\nu_{H \rightarrow x}(\mathcal{F}(H \rightarrow x)) < \frac{1}{2} \nu(\mathcal{F}) \right \} \\
        & \leq q\tau^{|H|}\nu(\mathcal{F}) + (1-q)\frac{1}{2}\nu(\mathcal{F}).
    \end{align*}
    On the other hand, we have $\EE \nu_{H \rightarrow \bx}(\mathcal{F}(H \rightarrow \bx)) = \nu(\ff)$ due to~\eqref{eq: double counting -- expectation}. Thus, we get $1 \leq q \tau^{|H|} + \frac{1}{2}(1-q)$. It implies
	$$q \geq \frac{1}{2\tau^{|H|}-1}.$$
	Substituting $\tau < \left(\frac{1+p}{2p}\right)^{\frac{1}{|H|}}$, we get the desired bound $q > p$.
\end{proof}

    Assume for a while, that pairs $(Z_i, x_i)$ obtained by the spread approximation lemma~\ref{lemma: spread approximation} are $t$-agreeing, i.e. for any $(Z_i, x_i), (Z_j, x_j)$ from $\{(Z_i, x_i)\}_{i = 1}^l$, we have $\agr_{{Z_i \cap Z_j}}(x_i, x_j) \ge t$. In the next section, we discuss a number of arguments due to papers~\cite{kupavskii_spread_2022} and~\cite{kupavskii2025lineardependenciespolynomialfactors}, which help to deal with this and weaker properties of the spread approximations and to prove that there exists a pair  $(T, x)$, $|T| = t$ and $x \in [m]^T$, such that for almost all codes $y \in \ff$, we have $y|_T = x$.

    \subsection{Simplification arguments}

    Simplification arguments are one of the key components of the spread approximation technique, but they vary from problem to problem~\cite{kupavskii_spread_2022, kupavskii2023erd, kupavskii2024almost, frankl2025hajnal, kupavskii2025lineardependenciespolynomialfactors, iarovikova2025complete}. All of them are designed for families of sets from an ambient family $\cA$ that meet some conditions. First, we list them.


    Say that a family of sets $\ff$ is $r$-spread if for any $Z$ we have
    \begin{align*}
        |\ff(Z)| \le r^{-|Z|} |\ff|.
    \end{align*}
    Then, we require the following assumption, which is a stronger version of Assumption 1~\cite{kupavskii2025lineardependenciespolynomialfactors}.

\begin{assumption}\label{assum1}
    A family $\mathcal{A}$ is $k$-uniform. Also, for any $X$ that is a subset of some set from $\mathcal{A}$ the family $\mathcal{A}(X)$ is $r$-spread. 
\end{assumption}

The second assumption is used when, for a $\tau$-homogenous family $\ff$, one needs to find a lot of sets $H$ in the shadow of $\cA$ such that the $\ff(H)$ is $\tau'$-homogeneous for some $\tau'$.

\begin{assumption}\label{assum2}
    For any set $S \in \partial_{\leqslant q}\mathcal{A}$, a family $\mathcal{F} \subset \mathcal{A}(S)$, and $h \le t - 1$, we have $$\frac{|\mathcal{F}|}{|\mathcal{A}(S)|} = \frac1{|\partial_h(\mathcal{A}(S))|}\sum_{H \in \partial_h(\mathcal{A}(S))} \frac{|\mathcal{F}(H)|}{|\mathcal{A}(H \cup S)|},$$ or, equivalently, for a random set $\mathbf{H}$ uniformly distributed on $\partial_h(\mathcal{A}(S))$, we have $\mu(\mathcal{F}) = \mathbb{E}\mu(\mathcal{F}(\mathbf{H}))$, where $\mu$ is the uniform measure.
\end{assumption}

We proceed with verifying that the necessary assumptions hold for the family $\cAcodes$ defined by~\eqref{eq: codes set family definition} with proper choice of $r, t, q, k$.

\begin{lemma}\label{assumptions_hold}
Assumptions \ref{assum1} and \ref{assum2} hold for the family $\cAcodes$, any positive real $r \leqslant m,$  integers $t \leqslant n+1, q \ge t$, and $k = n$.
\end{lemma}

\begin{proof}
   We start with verifying Assumption \ref{assum1}. Obviously, $\cAcodes$ considered as a subset of $[mn] \choose n$ is $n$-uniform. Let $X$ be a subset of some set from $\cAcodes$, $Y$ be a subset of some set from $\cAcodes(X)$, i.e., $X$ and $Y$ are restrictions $(Z, x)$, $(W, y)$ for some disjoint $Z, W\subset [n], x \in [m]^Z, y \in [m]^W$. We should prove that $$|\cAcodes(X \cup Y)| \leqslant r^{-|Y|}|\cAcodes(X)|.$$ It is equivalent to $$m^{n - |X| - |Y|} \le r^{-|Y|}\cdot m^{n - |X|},$$ which holds for all $r \le m$.
   
   Now we verify Assumption \ref{assum2}. Let $S = (Z, x)$ be a restriction, $|Z| \le q$ and let $\ff$ be a subfamily of $[m]^n(Z \to x)$. The expression to verify can be rewritten as follows $$\frac{|\ff|}{m^{n - |Z|}} = \frac1{{n - |Z| \choose h}m^h}\sum_{H \in \partial_h([m]^n(Z \to x))}\frac{|\ff(H)|}{m^{n - h - |Z|}},$$ which is equivalent to $$|\ff|\cdot{n - |Z| \choose h} = \sum_{H \in \partial_h([m]^n(Z \to x))}|\ff(H)|.$$ This equality holds since it is a double-counting of $h$-subsets of sets in $\ff$ with multiplicities.  
\end{proof}
Later we will only use $r = m$ since we have verified the $m$-spreadness and $k=n$ since we have verified $n$-uniformity.
Concerning the spread approximation discussed in the previous section, if restrictions $$\{(Z_i, x^i)\}_{i = 1}^l \subset \partial_{\le q} [m]^n$$
are $t$-agreeing, then the corresponding sets $\widetilde{Z}_i \in \partial_{\le q} \cAcodes$, $\widetilde{Z}_i \cap I_{j} = \{(j - 1)m + (x^j)_i\}$ for $j \in Z_i$ and $i = 1, \ldots, l$, are $t$-intersecting. We say that a $t$-intersecting family $\widetilde{\cS} = \{\widetilde Z_i\}$ is trivial, if there exists a set $\widetilde{T}$ of size $t$ such that $\widetilde{T} \subset \widetilde Z$ for all $\widetilde Z \in \widetilde\cS$.  It turns out that any non-trivial $t$-intersecting family of restrictions spans a much smaller number of sets than a single set of size $t$. This was shown  in the paper~\cite{kupavskii_spread_2022} under Assumption~\ref{assum1}. We use the following improved result due to~\cite{kupavskii2023erd}.

\begin{theorem}[Theorem 14~\cite{kupavskii2023erd}]
\label{theorem: simplification}
Let $\varepsilon \in (0; 1], n, r, q, t \ge 1$ be such that $\varepsilon r \ge 24 q$. Let $\cA \subset 2^{[n]}$ satisfy Assumption~\ref{assum1} with this $r$ and let $\cS \subset \binom{[n]}{\le q}$ be a non-trivial $t$-intersecting family. Then, there exists a set $T \in \partial_t \cA$ such that
\begin{align*}
    |\cA[\cS]| \le \varepsilon |\cA(T)|.
\end{align*}
\end{theorem}

We refer to Theorem~\ref{theorem: simplification} as a simplification argument, since it implies that if $\{(Z_i, x_i)\}_{i = 1}^l$ are restrictions of the spread approximation of a large enough family $\ff \subset [m]^n$, $n \ge \poly(t) \log m$, which avoids an agreement of size $t - 1$, then they can be simultaneously ``simplified'' to a restriction of size $t$. In particular, Theorem~\ref{theorem: simplification} yields the following corollary in our setting.

\begin{corollary}
\label{corollary: simplification argument}
Let $\varepsilon \in (0;1]$, $m, n, q, t \ge 1$ be such that $\varepsilon m \ge 24 q$. Let $\cS = \{(Z_i, x_i)\}_{i = 1}^l \subset \partial_{\le q} [m]^n$ be a family of $t$-agreeing restrictions. Suppose that there is no restriction $(T, x)$, $|T| = t$ and $x \in [m]^T$, such that $T \subset Z_i$ and $x_i|_T = x$ for all $i = 1, \ldots, l$. Then, we have
\begin{align*}
    |[m]^n[\cS]| \le \varepsilon m^{n - t}.
\end{align*}
\end{corollary}

In the regime $n \ge \poly(t) \log m$, application of Corollary~\ref{corollary: simplification argument} almost completes the proof. It implies that, for an extremal family $\ff \subset [m]^n$ that does not contain a $(t - 1)$-agreement, there is a restriction $(T, x)$, $|T| = t$, such that $\ff$ is almost fully contained in $\ff[T \to x]$. Applying Lemma~\ref{lemma: unbalanced cross matching}, we will deduce that $\ff \setminus \ff[T \to x]$ is empty.

Next, we move on to the discussion of the regime $m \ge n^{Ct}$ and tools required for its analysis, which were developed in the paper of Kupavskii and the second author~\cite{kupavskii2025lineardependenciespolynomialfactors}.
First, we give some prelimanries concerning sunflowers and $(\cS, s, t)$-systems. Following Definition~8 of~\cite{kupavskii2025lineardependenciespolynomialfactors}, we will name a family $\cS = \{(Z_i, x_i)\}_{i = 1}^l$ obtained from the spread approximation in Lemma~\ref{lemma: spread approximation}.

\begin{definition}
    A family $\mathcal{S} = \{(Z_1, x_1), \ldots, (Z_l, x_l)\}$ where $Z_i \subset [n]$ and $x_i \in [m]^{Z_i}$ is a $\tau$-homogenous core of $\mathcal{F} \subset [m]^n$ if $\mathcal{F}$ admits a decomposition $$\mathcal{F} = \bigsqcup_{S \in \mathcal{S}}\mathcal{F}_S$$ where for all $S = (Z, x) \in \mathcal{S}$ we have $\mathcal{F}_S = \mathcal{F}_S[Z \to x]$ and families $\mathcal{F}_S(Z\to x)$ are $\tau$-homogenous.
\end{definition}

We will sometimes treat $x_i$ as elements of $\partial_{\le n} \cC \subset  \partial_{\le n}{mn \choose n}$. 

To tackle the case $m \ge n^{Ct}$, we will adapt the technique of $(\cS, s, t)$-systems of~\cite{kupavskii2025lineardependenciespolynomialfactors} which was developed for systems without sunflowers. Hence, we need to introduce the definition of a sunflower.
\begin{definition}
    A system of sets $F_1, F_2, \ldots, F_s$ is said to form a sunflower if and only if for all $i, j \in [s], i \neq j$ we have $$F_i \cap F_j = \bigcap_{k \in [s]}F_k.$$
\end{definition}

The set $\bigcap_{k \in [s]}F_k$ is called \textit{core}, and sets $F_1, F_2, \ldots, F_s$ are called \textit{petals}. After that we introduce the definition of $(\mathcal{S}, s, t)$-system.

\begin{definition}
    Let $\mathcal{A}$ be an arbitrary family, $\mathcal{S} \subset \bigcup_{p=t}^q\partial_p\mathcal{A}$ be a family that does not contain a sunflower with $s$ petals and core of size $t-1$. We say that a family $\mathcal{B} \subset \mathcal{A}$ is an $(\mathcal{S}, s, t)$-system, if $\mathcal{B}$ can be decomposed into disjoint families $\mathcal{B}_S \vee \{S\}, S \in \mathcal{S}$ and $\mathcal{B}_S \subset \mathcal{A}(S)$, such that for any $s$ sets $S_1, S_2, \ldots, S_s \in \mathcal{S}$ the following properties hold.

    \begin{itemize}
    \item $S_1, S_2, \ldots, S_s$ do not form a sunflower with the core of size exactly $t-1$;
    \item if $S_1, S_2, \ldots, S_s$ form a sunflower with core of size at most $t-2$, then for any $F_i \in \mathcal{B}_{S_i}, i \in [s]$ we have $|\bigcap_{i=1}^sF_i| \leqslant t - |S_1\cap S_2| - 2$.
    \end{itemize}
\end{definition}

It is easy to check that our problem corresponds to forbidding a sunflower with $2$ petals and the core of size $t - 1$ in a subfamily of $\cC$. Hence, we may apply the main results of~\cite{kupavskii2025lineardependenciespolynomialfactors} on $(\cS, s, t)$-systems:
\begin{enumerate}
	\item the decomposition of a family avoiding $(t - 1)$-intersection into $(\cS, 2, t)$-system;
	\item the simplification of $(\cS, s, t)$-system into a moderately small number of incomplete $t$-stars.
\end{enumerate}


\begin{lemma}[Lemma 56 of~\cite{kupavskii2025lineardependenciespolynomialfactors}]
\label{get_sst}
Let $\ff \subset [m]^n$ be a code avoiding $(t - 1)$-agreement, $t \le n + 1, q \ge t, n\ge q + t - 1$. Let $\mathcal{S}$ be a $\tau$-homogenous core of $\ff$ such that for all $(Z, x) = S \in \mathcal{S}$ we have $|Z| \le q$. Let $\{\ff_S\}_{S\in \mathcal{S}}$ be the corresponding decomposition of $\ff$. Assume that $m > \max\{2^{15}\lceil\log_2n\rceil, 2q\}\cdot 2\alpha^{1-t}\tau^t$ for some $\alpha \in (0, (4n)^{-1}].$ Then for all $(Z,x) = S \in \mathcal{S}$ there exist families $\mathcal{U}_S\subset \mathcal{F}_S = \ff_S[Z \to x]$ such that
\begin{enumerate}
        \item $|\mathcal{U}_S| \ge (1 - 2\alpha n)|\mathcal{F}_S|$ for each $S \in \mathcal{S}$;
        \item $\mathcal{S}$ avoids intersection $t - 1$; if $S_1 = (Z_1, x_1), S_2 = (Z_2, x_2)\in \mathcal{S}$ agree on at most $t - 2$ coordinates, then for any $y_1 \in \mathcal{U}_{S_1}, y_2 \in \mathcal{U}_{S_2}$, we have $\agr_{[n] \setminus (Z_1 \cup Z_2)}(y_1, y_2) \le t - |S_1\cap S_2| - 2$;
        \item let $\tau' \le \tau$ be the minimal homogenity of $\mathcal{F}_S$, then for any $h < n$ we have $$|\partial_h(\mathcal{U}_S)| \ge \left(\frac{\tau'}{1 - 2\alpha n}\right)^{-h}|\partial_h([m]^n(S))|.$$
    \end{enumerate}
\end{lemma}

It is easy to check that $\cU_{S}, S \in \cS,$ form a $(\cS, 2, t)$-system. Moreover, the extra property 3 that bounds the shadow of each $\cU_S$ below allows to prove the following lemma.

\begin{lemma}[Proposition 57 of~\cite{kupavskii2025lineardependenciespolynomialfactors}]
\label{get_T}
    Let $\mathcal{U} \subset [m]^n$ be an $(\mathcal{S}, 2, t)$-system. Let $\{\mathcal{U}_S\}_{S \in \mathcal{S}}$ be the decomposition of $\mathcal{U}$ from the definition of an $(\mathcal{S},s, t)$-system and $\lambda$ be a positive real number such that $$|\partial_{t-1}\mathcal{U}_S| \ge \lambda|\partial_{t-1}[m]^n|$$ for every $S\in \mathcal{S}$. Suppose that $m \ge 8q^3$. Then there exists a family $\hat{\mathcal{T}} \subset \partial_t[m]^n$, such that $$|\hat{\mathcal{T}}| \le \frac1\lambda(1 +2\ln(\lambda|\partial_{\le q}[m]^n|))$$ and $$\sum_{S \in \mathcal{S}\setminus \mathcal{S}[\hat{\mathcal{T}}]}|\mathcal{U}_S| \le \frac{32q^3}{\lambda m}\ln(\lambda|\partial_{\le q}[m]^ n|)\cdot m^{n-t}.$$
\end{lemma}

Using spread approximation Lemma~\ref{lemma: spread approximation}, one can decompose a $(t - 1)$-avoiding code $\ff$ into homogeneous families $\ff_S, S \in \cS$, apply Lemma~\ref{get_sst} and then deduce that $\ff$ can be decomposed into relatively small number of incomplete $t$-stars $\ff_T, T \in \widehat{\cT}$ up to some small remainder. This is a quite rich structure which we further refine. One of the key components of this part of our analysis is the Kruskal--Katona theorem for codes.

\begin{theorem}[\cite{frankl_shadows_1988, london1994new}]
\label{KK_direct}
 Suppose $\ff \subset [m]^n$ and $\delta \ge 0$ is defined by $|\ff| = \delta m^n$. Then for all $1\le l \le n$ one has $|\partial_l\ff| \ge \delta^{\frac ln}|\partial_l[m]^n|$.
\end{theorem}

However, instead of the above theorem we often use its contraposition.

\begin{corollary}\label{KK_contra}
    Suppose $\ff \subset [m]^n$ and for some integer $l < n$ and some $\delta \ge 0$ we have $|\partial_l\ff| \le \delta |\partial_l[m]^n|$. Then one has $|\ff| \le \delta^{\frac nl}m^n$.
\end{corollary}

Applying Corollary~\ref{KK_contra} together with some geometrical scaling decomposition, we show that an extremal $(t - 1)$-avoiding code $\ff$ is almost fully contained in some $t$-star. Then, we apply Corollary~\ref{KK_contra} again and show that $\ff$ must be a $t$-star, completing the case $m \ge n^{Ct}$.


\section{Cross-disagreement in homogeneous families}
\label{section: cross-disagreement in homogeneous families}

    In this section we prove a sufficient condition of existing of a cross-disagreement in homogenous families. The key ingredient of the proof is the following lemma.
	
	\begin{lemma}
		\label{lemma: boosting measure step}
		Let $m, s$ be integer numbers such that $s | m$ and $s \geq 4$. Let $\nu$ be an $s$-balanced product measure on $[m]^n$ and $\mathcal{F}$ be a $\tau$-homogenous family with respect to $\nu$. Then there exists a balanced gluing $\pi \in \Pi_{m, m/s}^{\otimes n}$ such that $$\nu^{\pi}(\ff^{\pi}) \ge \nu^{1-c(\tau)}(\ff),$$ where $$c(\tau) = \frac{\ln \frac{6}{5}}{8 \ln \left(\frac{2^7}{\ln 4}\right) + 4 \ln \tau}.$$
	\end{lemma}

    We will apply this lemma iteratively to transform exponentially small measures into constant measures.

	\begin{proof}
		By Corollary \ref{corollary: stab measure upgrade} there exists $\pi \in \Pi_{m, m/s}^{\otimes n}$ such that $\nu^{\pi}(\ff^{\pi}) \ge \frac{\nu^2(\ff)}{\Stab_{5/6}(\ff)}$. Thus, it is sufficient to get an upper bound on $\Stab_{5/6}(\ff)$. By Proposition~\ref{proposition: smaller rho stability upper bound}, an upper bound on $\Stab_{\rho_0}(\ff)$ for any $\rho_0 > 0$ implies an upper bound on $\Stab_{5/6}(\ff)$.
		
		We use Theorem \ref{theorem: hypercontractivity} to bound $\Stab_{\rho_0}(\ff)$ for $\rho_0 = \frac{\ln 4}{2^{7}\tau^{\frac{1}{2}}}$. By Theorem \ref{theorem: hypercontractivity} for $r = \tau^{\frac{1}{2}}, q=4, \rho =\rho_0$, we get $||T_{\rho_0}f||_4 \leq ||f||_2$. By the H\"older inequality, we get
		\begin{align*}
			\Stab_{\rho_0}(f) = \langle f, T_{\rho_0} f \rangle \le \Vert T_{\rho_0} f \Vert_{4} \Vert f \Vert_{4/3} \le \Vert f \Vert_{2} \Vert f \Vert_{4/3} = \nu^{5/4}(\ff).
		\end{align*}
		Let $t$ be a number such that $t = 2^d, d \in \mathbb{N}$ and $$\frac{\ln\left(\frac{2^7}{\ln 4}\right) + \frac{1}{2} \ln \tau}{\ln \frac{6}{5}} \leq t < 2\frac{\ln\left(\frac{2^7}{\ln 4}\right) + \frac{1}{2} \ln \tau}{\ln \frac{6}{5}}.$$
        $\Stab_{\rho}(f)$ is monotone as a function of $\rho$ because $\Stab_{\rho}(f) = \sum_{S \subset [n]}\rho^{|S|}||f^{=S}||_2^2$, where $f^{=S}, S \subset [n],$ is the Efron--Stein decomposition of $f$.\footnote{For more details on the Efron--Stein decomposition, we refer the reader to Chapter  8 of~\cite{o2014analysis}}
		Since $(\frac{5}{6})^t \le\rho_0$, we have $\Stab_{(\frac{5}{6})^t}(f) \leq \Stab_{\rho_0}(f) \leq \nu^{5/4}(\ff)$. By Proposition~\ref{proposition: smaller rho stability upper bound}, we get
		$$\Stab_{5/6}(f) \leq \Vert f \Vert^{2 (1 - 1/t)}_2 \Stab_{(\frac{5}{6})^t}(f)^{1/t} \leq \nu^{1-\frac{1}{t}}(\ff)\nu^{\frac{5}{4t}}(\ff) = \nu^{1+\frac{1}{4t}}(\ff) \leq \nu^{1+c(\tau)}(\ff).$$
		Finally, we combine this bound with the inequality $\nu^{\pi}(\ff^{\pi}) \ge \frac{\nu^2(\ff)}{\Stab_{5/6}(\ff)}$ and get the desired bound $\nu^{\pi}(\ff^{\pi}) \ge \nu^{1-c(\tau)}(\ff)$.
	\end{proof}

    Now we are ready to proof a measure boosting lemma. The lemma is a counterpart of Lemma $6.11$ from \cite{keevash2023forbidden}.
	
	\begin{lemma}
		\label{lemma: boosting measure}
		Let $b \geq 2$ be an integer and $\nu$ be a $b$-balanced product measure on $[m]^n$. Let $\mathcal{F} \subset [m]^n, \nu(\mathcal{F}) = \alpha$. Let $c(\tau)$ be as in Lemma \ref{lemma: boosting measure step},
		$$c_1(\tau) = 4 + \frac{2 \ln \ln 2}{\ln (1 - c(\tau))}$$
		and
		$$c_2(\tau) = -\frac{2}{\ln(1-c(\tau))}.$$
		For any $\tau > 1$ such that
		$$m > b^{c_1(\tau)}\left(\ln \frac{1}{\alpha}\right)^{c_2(\tau)\ln b},$$ and $n > \frac{\ln \alpha^{-1}}{\ln \tau}$ there are $m' \in \mathbb{N}, \pi \in \Pi_{m, m', b}^{\otimes n}, R \subset [n]$ and $x \in [m']^{R}$ such that
		$$m' \geq \frac{m}{ b^{c_1(\tau)}\left(\ln \frac{1}{\alpha}\right)^{c_2(\tau)\ln b}},$$
		$ |R| \le \frac{\ln \alpha^{-1}}{\ln \tau}$, $\ff^{\pi}(R \rightarrow x)$ is $\tau$-homogenous with respect to $\nu^{\pi}_{R\rightarrow x}$ and $\nu^{\pi}_{R\rightarrow x}(\ff^{\pi}(R \rightarrow x)) \ge \frac{1}{2}$.
	\end{lemma}

    The main idea of the proof it to apply iteratively Lemma \ref{lemma: boosting measure step}.

\begin{proof}
	Let $m_0$ be such a number, that $m_0 = b^k$ for some $k \in \mathbb{N}$ and $m/b < m_0 \leq m$. Let $\pi_0 \in \Pi_{m, m_0, b}^{\otimes n}$ be an arbitrary $b$-balanced gluing, $\nu'_0 = \nu^{\pi_0}$ and $\mathcal{F}'_0 = \mathcal{F}^{\pi_0}$. Note, that $\nu'_0$ is $b^2$-balanced, sunce $\nu$ is $b$-balanced and $\pi_0$ is $b$-balanced. Let $Z_0 \subset [n]$ be a maximal set such that for some $r_0 \in [m_0]^{Z_0}$ we have
	$(\nu'_0)_{Z_0 \to r_0}(\mathcal{F}'_0(Z_0 \to r_0)) \ge \tau^{|Z_0|} \nu'_0(\mathcal{F}'_0)$. Let $\nu_0 = (\nu'_0)_{Z_0 \to r_0}$ and $\mathcal{F}_0 = \mathcal{F}'_0(Z_0 \to r_0)$. By Lemma \ref{lemma: homogenous restriction}, the family $\mathcal{F}_0$ is $\tau$-homogenous with respect to $\nu_0$.
	
	For $i \geq 1$ we iteratively construct tuples $(m_i, \pi_i, \nu_i, Z_i, r_i, \mathcal{F}_i)$ such that $m_i = \frac{m_0}{b^{2i}}$, $\pi_i \in \Pi_{m, m_i, b}^{\otimes n}$, $Z_i \subset [n]$, $r_i \in [m_i]^{Z_i}$, $\nu_i = \nu^{\pi_i}_{Z_i\rightarrow r_i}$, $\mathcal{F}_i = \mathcal{F}^{\pi_i}(Z_i \to r_i)$,
	$$\nu_i(\mathcal{F}_i) \geq \max(\tau^{|Z_i|}\nu(\mathcal{F}), \nu(\mathcal{F})^{\left(1-c(\tau)\right)^i})$$
	and $\mathcal{F}_i$ is $\tau$-homogenous with respect to $\nu_i$. Note that $(m_0, \pi_0, \nu_0, Z_0, r_0, \mathcal{F}_0)$ satisfies the conditions. We construct the tuples by the following iterative procedure.
	
	\begin{enumerate}
		\item If $\nu_{i-1}(\mathcal{F}_{i-1}) \geq \frac{1}{2}$ or $m_{i-1} \leq b$, stop.
		\item Let $m_i = \frac{m_{i-1}}{b^2}$. Note that $b^2 | m_{i-1}$, since $m_{i-1}$ is an integer power of $b$ and $m_{i-1} > b$. Let $\pi'_i \in \Pi_{m_{i-1}, m_i}^{\otimes ([n]\setminus Z_{i-1})}$ be a gluing obtained by Lemma \ref{lemma: boosting measure step} with $s = b^2$ and $\nu = \nu_{i - 1}$. Define $\nu'_i = \nu_{i-1}^{\pi'_i}$ and $\mathcal{F}'_i = \mathcal{F}_{i-1}^{\pi'_i}$. By Lemma \ref{lemma: boosting measure step}, we get $\nu'_i(\mathcal{F}'_i) \geq \nu_{i-1}(\mathcal{F}_{i-1})^{1-c(\tau)}$.
		\item Consider an arbitrary balanced gluing $\pi_i'' \in \Pi_{m_{i-1}, m_i}^{\otimes Z_{i-1}}$. Let $\pi_i = (\pi_i' \otimes \pi_i'') \circ \pi_{i-1}$, $r''_i = \pi_i''(r_{i-1})$, $\mathcal{F}''_i = \mathcal{F}^{\pi_i}(Z_{i-1} \to r''_i)$. Since $\mathcal{F}'_i \subset \mathcal{F}''_i$, we get $\nu'_i(\mathcal{F}''_i) \geq \nu_{i-1}(\mathcal{F}_{i-1})^{1-c(\tau)}$. Moreover, $\pi_i$ is $b$-balanced, because $\pi_{i-1}$ is $b$-balanced and $(\pi_i' \otimes \pi_i'')$ is $1$-balanced.
		\item Let $Z'_i \subset [n] \setminus Z_{i - 1}$ be maximal such that for some $r'_i \in [m_i]^{Z'_i}$ we have
        \begin{align}
        \label{eq: definition of Z prime-i}
            (\nu'_i)_{Z'_i \to r_i'}(\mathcal{F}''_i(Z'_i \to r'_i)) \ge \tau^{|Z'_i|} \nu'_i(\mathcal{F}''_i).    
        \end{align}
		 Let $\nu_i = (\nu'_i)_{Z'_i \to r'_i}$ and $\mathcal{F}_i = \mathcal{F}''_i(Z'_i \to r'_i)$. By Lemma \ref{lemma: homogenous restriction}, the family $\mathcal{F}_i$ is $\tau$-homogenous with respect to $\nu_i$. Finally, let $Z_i = Z_{i-1}\sqcup Z'_i$ and $r_i = (r'_{i}, r''_{i})$.
	\end{enumerate}

	By construction, we get $m_i = \frac{m_0}{b^{2i}}$, $\pi_i \in \Pi_{m, m_i, b}^{\otimes n}$, $Z_i \subset [n]$, $r_i \in [m_i]^{Z_i}$, $\nu_i = \nu^{\pi_i}_{Z_i\rightarrow r_i}$, $\mathcal{F}_i = \mathcal{F}^{\pi_i}(Z_i \to r_i)$. By induction,
	$$\nu_i(\mathcal{F}_i) \geq \nu'_i(\mathcal{F}''_i) \geq \nu_{i-1}(\mathcal{F}_{i-1})^{1-c(\tau)} \geq \left(\nu(\mathcal{F})^{\left(1-c(\tau)\right)^{i-1}}\right)^{1-c(\tau)} = \nu(\mathcal{F})^{\left(1-c(\tau)\right)^i}$$
	and
	$$\nu_i(\mathcal{F}_i) \geq \tau^{|Z'_i|}\nu'_i(\mathcal{F}''_i) \geq \tau^{|Z'_i|}\nu_{i-1}(\mathcal{F}_{i-1}) \geq \tau^{|Z'_i|}\tau^{|Z_{i-1}|}\nu(\mathcal{F}) = \tau^{|Z_i|}\nu(\mathcal{F}).$$

    Let $I$ be the value of $i$ when the procedure terminates. Let $m' = m_{I-1}, \pi = \pi_{I-1}, R = Z_{I-1}, x = r_{I-1}$.  If $I = 1$, then $m_0 \ge m/b^2 > b$, so $\nu_0(\ff_0) \ge 1/2$.  If $I \ge 2$, we will show that the procedure terminates by the condition $\nu_{I - 1}(\ff_{I - 1}) \ge 1/2$ as well.
	 Since $\alpha^{\left(1-c(\tau)\right)^{I-2}} \leq \nu_{I-2}(\mathcal{F}_{I-2}) < \frac{1}{2} $, we get $I \leq 2 + \frac{\ln\ln 2}{\ln(1 - c(\tau))} - \frac{\ln \ln \frac{1}{\alpha}}{\ln(1 - c(\tau))}$. Therefore, $$m_{I-1} \geq \frac{m}{b^{2I-1}} = b\frac{m}{b^{2I}} \geq b\frac{m}{b^{c_1(\tau)}\left(\ln \frac{1}{\alpha}\right)^{c_2(\tau)\ln b}} > b.$$
	Thus, we indeed have $$\nu^{\pi}_{R\rightarrow x}(\ff^{\pi}(R \rightarrow x)) = \nu_{I-1}(\mathcal{F}_{I-1}) \geq \frac{1}{2}.$$
	It remains to bound $|R|$. Definition~\eqref{eq: definition of Z prime-i} ensures that $ \tau^{|R|} \alpha \le \nu_{I-1}(\mathcal{F}_{I-1}) \leq 1$, so $|R| \le \frac{\ln \alpha^{-1}}{\ln \tau}$.
\end{proof}

We are ready to prove a sufficient condition on existing of cross-disagreement in homogenous families.

\begin{lemma}
	\label{lemma: cross matching for significant measure}
	Let $\mathcal{F}_1 \subset [m]^{[n]\setminus H_1}, \mathcal{F}_2 \subset [m]^{[n]\setminus H_2}$. Let $b \geq 2$ and $\nu$ be a $b$-balanced product measure on $[m]^n$. Let $\alpha_i = \nu_{[n]\setminus H_i}(\mathcal{F}_i)$. Let $\mathcal{F}_2$ be $\tau$-homogenous with respect to $\nu_{[n] \setminus H_2}$. Let
	$$C_1(b) = \frac{4}{\ln 2}b^{3c_1(2)+4}, \qquad C_2(b) = 2c_2(2)\ln b + 1,$$
	where $c_1(\tau)$ and $c_2(\tau)$ are from Lemma \ref{lemma: boosting measure}. If
	$$m > \tau C_1(b) \left(\max\left(1, \ln(\alpha_1^{-1})\right)\max\left(1, \ln(2\alpha_2^{-1})\right)\right)^{C_2(b)}$$
	and $n > \frac{\ln(\alpha_1^{-1})+\ln(2\alpha_2^{-1})}{\ln 2} + |H_1| + |H_2|$, then there exist $x_i \in \mathcal{F}_i$ such that $\agrparam{[n]\setminus (H_1 \cup H_2)}{x_1}{x_2} = 0$. 
\end{lemma}

The formal proof of Lemma \ref{lemma: cross matching for significant measure} is quite technical; therefore, we give its sketch first. We make several modifications of families $\mathcal{F}_1$ and $\mathcal{F}_2$ to get families with much bigger measure on the same groud set $[m']^{n'}$ with $m' \ll m$, such that if the modified families contain cross-disagreement than the initial families also contain cross-disagreement. Then we use Lemma \ref{lemma: hoffman bound} to find cross-disagreement in the modified families. The proof is divided into the following steps.

\textbf{Step $1$.} We use Lemma \ref{lemma: boosting measure} to get a restriction $R \rightarrow x$ and a gluing $\pi$ such that $\mathcal{F}'_1 = \ff_1^{\pi}(R \rightarrow x)$ has $\nu^{\pi}_{R \rightarrow x}$-measure at least $\frac{1}{2}$.

\textbf{Step $2$.} We erase from $\mathcal{F}_2$ sets, that agree after the gluing $\pi$ with $x$ on coordinates in $R$. Lemma \ref{lemma: homogenous avoiding} implies that the measure of $\mathcal{F}_2$ decreases insignificantly. We denote by $\mathcal{F}'_2$ sets from $\mathcal{F}_2$, that were not erased in this step.

\textbf{Step $3$.} We transform the family $\mathcal{F}'_2 \subset [m]^{[n] \setminus H_2}$ into a family $\mathcal{F}''_2 \subset [m']^{[n] \setminus (H_2 \cup R)}$ without decreasing measure. First, using Claim \ref{claim: averaging argument}, we find a restriction on $R \setminus H_2$ that does not decrease measure. Second, we apply a gluing $\pi$ obtained in the first step to this restriction. Claim \ref{claim: measure consistency after a gluing} ensures that the gluing does not decrease measure. Note, that in this step we lose homogenity of $\mathcal{F}_2$. However, Lemma \ref{lemma: boosting measure} ensures homogenity of $\mathcal{F}'_1$.

\textbf{Step $4$.} We repeat steps $1$-$3$, switching families. That is, we boost the measure of $\mathcal{F}''_2$ and ensure that the measure of $\mathcal{F}'_1$ decreases insignificantly. After this step, we have families $\mathcal{F}'''_1 \subset [m'']^{[n] \setminus (H_1 \cup R \cup R')}$ and $\mathcal{F}'''_2 \subset [m'']^{[n] \setminus (H_2 \cup R \cup R')}$ with constant measures such that the existence of cross-disagreement in $\mathcal{F}'''_1$ and $\mathcal{F}'''_1$ implies the existence of cross-disagreement in the initial families.

\textbf{Step $5$.} At this step, the families $\mathcal{F}'''_1$ and $\mathcal{F}'''_2$ are in different sets of coordinates. We use Claim \ref{claim: averaging argument} to obtain restrictions $\mathcal{F}''''_1$ and $\mathcal{F}''''_2$ on the same ground set without decreasing the measure.

\textbf{Step $6$.} We use Lemma \ref{lemma: hoffman bound} to find the cross-disagreement in $\mathcal{F}''''_1$ and $\mathcal{F}''''_2$.

Next, we provide a formal proof of Lemma \ref{lemma: cross matching for significant measure}.

\begin{proof}
	Let $\beta_1 = \min(\alpha_1, e^{-1}), \beta_2 = \min(\frac{\alpha_2}{2}, e^{-1})$. Then, we have $\nu_{[n]\setminus H_1}(\mathcal{F}_1) \geq \beta_1, \nu_{[n]\setminus H_2}(\mathcal{F}_2) \geq 2\beta_2, \ln(\beta_i^{-1}) \geq 1$, and the requirement on $m$ is equivalent to
	$$m > \tau C_1(b) \left(\ln(\beta_1^{-1}) \ln(\beta_2^{-1})\right)^{C_2(b)}.$$

    \textbf{Step $1$.}
	By Lemma \ref{lemma: boosting measure} for $\mathcal{F}=\mathcal{F}_1$ and $\tau = 2$, we obtain $m' \in \mathbb{N}, \pi_{-H_1} \in \Pi_{m, m', b}^{\otimes [n] \setminus H_1}, R \subset [n]\setminus H_1$ and $x \in [m']^{R}$ such that
	$$m' \geq \frac{m}{ b^{c_1(2)}\left(\ln \frac{1}{\beta_1}\right)^{c_2(2)\ln b}} \geq \frac{4}{\ln 2}\tau b^{2c_1(2)+4} \left(\ln \left(\beta_2^{-1}\right)\right)^{2c_2(2)\ln b + 1}\left(\ln \beta_1^{-1}\right),$$
	$|R| \le \frac{\ln \alpha_1^{-1}}{\ln 2}$, $\ff_1^{\pi_{-H_1}}(R \rightarrow x)$ is $2$-homogenous with respect to $(\nu_{[n]\setminus H_1})^{\pi_{-H_1}}_{R\rightarrow x}$ and $$(\nu_{[n]\setminus H_1})^{\pi_{-H_1}}_{R\rightarrow x}(\ff_1^{\pi_{-H_1}}(R \rightarrow x)) \ge \frac{1}{2}.$$ Let $\pi \in \Pi_{m, m', b}^{\otimes n}$ be arbitrary such that $\pi_{[n]\setminus H_1} = \pi_{-H_1}$. Let $\mathcal{F}'_1 = \ff_1^{\pi}(R \rightarrow x)$ and $\nu' = \nu^{\pi}_{[n]\setminus R}$.

    \textbf{Step $2$.}
	Since $\pi$ is $b$-balanced, we get $|\pi_{i}^{-1}(x_i)| \leq \frac{bm}{m'}$ for every $i \in [n]$. Let $\mathcal{F}'_2$ be the faimily of sets that do not agree after the gluing $\pi$ with $x$ on coordinates in $R$, i. e. $$\mathcal{F}'_2 = \{y \in \mathcal{F}_2 | \forall z \in \pi_{R \setminus H_2}^{-1}(x|_{R \setminus H_2}) \, \agr_{R \setminus H_2}(z, y) = 0\}.$$ By Lemma \ref{lemma: homogenous avoiding}, we get
	$$\nu_{[n]\setminus H_2}(\mathcal{F}'_2) \geq \left(1 - \frac{|R|\frac{bm}{m'} \tau b}{m}\right) \nu_{[n]\setminus H_2}(\mathcal{F}_2) \geq \left(1 - \frac{\tau b^2 \ln \alpha_1^{-1}}{m' \ln 2}\right)\alpha_2 \geq \frac{1}{2}\alpha_2.$$
	Since $\beta_2 \leq \frac{1}{2}\alpha_2$, we have $\nu_{[n]\setminus H_2}(\mathcal{F}'_2) \geq \beta_2$.

    \textbf{Step $3$.}
	Take arbitrary $y \in [m']^{R\setminus H_2}$ such that $\nu^{\pi}_{[n] \setminus (R \cup H_2)}({\mathcal{F}'_2}^{\pi}(R\setminus H_2 \rightarrow y)) \geq \nu_{[n]\setminus H_2}^{\pi}({\mathcal{F}'_2}^{\pi})$. Note, that such $y$ exists by  Claim~\ref{claim: averaging argument}, and $\agr(x, y) = 0$ because for all $y$ with $\agr(x, y) \neq 0$ we have ${\mathcal{F}'_2}^{\pi}(R \setminus H_2 \rightarrow y) = \emptyset$. Let $\mathcal{F}''_2 = {\mathcal{F}'_2}^{\pi}(R \setminus H_2 \rightarrow y)$. By Claim \ref{claim: measure consistency after a gluing}, we get
	$$\nu'_{[n] \setminus (R \cup H_2)}(\mathcal{F}''_2) = \nu^{\pi}_{[n] \setminus (R \cup H_2)}(\mathcal{F}''_2) \geq \nu^{\pi}_{[n]\setminus H_2}({\mathcal{F}'_2}^{\pi}) \geq \nu_{[n]\setminus H_2}(\mathcal{F}'_2) \geq \frac{1}{2}\alpha_2 \geq \beta_2.$$
	For any $x' \in \mathcal{F}'_1, y' \in \mathcal{F}''_2$ there exist $x \in \mathcal{F}_1, y \in \mathcal{F}_2$ such that
    $$\agrparam{[n]\setminus (H_1 \cup H_2)}{\pi(x)}{\pi(y)} = \agrparam{[n]\setminus (H_1 \cup H_2 \cup R)}{x'}{y'}$$
    and, therefore,
    $$\agrparam{[n]\setminus (H_1 \cup H_2)}{x}{y} \leq \agrparam{[n]\setminus (H_1 \cup H_2 \cup R)}{x'}{y'}.$$
    Thus, it is sufficient to find $x' \in \mathcal{F}'_1, y' \in \mathcal{F}''_2$ with 
    $\agrparam{[n]\setminus (H_1 \cup H_2 \cup R)}{x'}{y'} = 0$.

    \textbf{Step $4$.}
	Since $\nu$ is $b$-balanced and $\pi \in \Pi_{m, m', b}^{\otimes n}$, by Claim \ref{claim: balanced consistency}, $\nu^{\pi}$ is $b^2$-balanced and therefore $\nu'$ is $b^2$-balanced. We apply Lemma \ref{lemma: boosting measure} for $\mathcal{F} = \mathcal{F}''_2$, $b=b^2$, $\nu = \nu'$ and $\tau = 2$ and obtain $m'' \in \mathbb{N}, \pi'_{-H_2} \in \Pi_{m', m'', b^2}^{\otimes [n] \setminus (R \cup H_2)}, R' \subset [n] \setminus (R \cup H_2)$ and $y' \in [m'']^{R'}$ such that
	$$m'' \geq \frac{m'}{ b^{2c_1(2)}\left(\ln(\beta_2^{-1})\right)^{2c_2(2)\ln b}} \geq \frac{4}{\ln 2} b^{4}\ln(\beta_2^{-1}),$$
	$|R'| < \frac{\ln \left(2\alpha_2^{-1}\right)}{\ln 2}$, and $(\nu')^{\pi'_{-H_2}}_{R'\rightarrow y'}({\mathcal{F}_2''}^{\pi'_{-H_2}}(R' \rightarrow y')) \ge \frac{1}{2}$. Let $\pi' \in \Pi_{m, m', b}^{\otimes [n] \setminus R}$ be arbitrary such that $\pi'_{[n]\setminus (R \cup H_2)} = \pi'_{-H_2}$. Let $\mathcal{F}'''_2 = {\mathcal{F}_2''}^{\pi'}(R' \rightarrow y')$ and $\nu'' = \nu'^{\pi'}_{R' \rightarrow y'} = \nu'^{\pi'}_{[n] \setminus (R \cup R')}$.
	
	Since $\pi'$ is $b^2$-balanced, we get $|(\pi'_{i})^{-1}(y'_i)| \leq \frac{b^2m'}{m''}$ for all $i \in R'$. Let $\mathcal{F}''_1$ be the family of sets from $\ff'_1$ that do not agree after the gluing $\pi'$ with $y'$ on coordinates in $R'$, i. e.
    $$\mathcal{F}''_1 = \{x \in \mathcal{F}'_1 | \forall z \in \left(\pi_{R'\setminus H_1}'\right)^{-1}(y') \, \agrparam{R'\setminus H_1}{z}{x} = 0\}.$$
    By Lemma \ref{lemma: homogenous avoiding}, we get
	$$\nu'_{[n] \setminus (R \cup H_1)}(\mathcal{F}''_1) \geq \left(1 - \frac{2|R'|\frac{b^2m'}{m''} b^2}{m'}\right) \nu'_{[n] \setminus (R \cup H_1)}(\mathcal{F}'_1) \geq \left(1 - \frac{2 b^4 \ln \left(\beta_2^{-1}\right)}{m'' \ln 2}\right)\cdot \frac{1}{2} \geq \frac{1}{4}.$$
	
	Take arbitrary $x' \in [m'']^{R'\setminus H_1}$  such that $\nu'^{\pi'}_{[n] \setminus (R \cup R' \cup H_1)}({\mathcal{F}''_1}^{\pi'}(R' \setminus H_1 \rightarrow x')) \geq \nu'^{\pi'}_{[n]\setminus (R \cup H_1)}({\mathcal{F}_1''}^{\pi'})$. Note, that such $x'$ exists by Claim~\ref{claim: averaging argument} and $\agr(x', y') = 0$, because for all $x'$ with $\agr(x', y') \neq 0$ we have ${\mathcal{F}_1''}^{\pi'}(R' \setminus H_1 \rightarrow x') = \emptyset$. Let $\mathcal{F}'''_1 = {\mathcal{F}_1''}^{\pi'}(R' \setminus H_1 \rightarrow x')$. By Claim \ref{claim: measure consistency after a gluing} we get
	$$\nu''_{[n] \setminus (R \cup R' \cup H_1)}(\mathcal{F}'''_1) = \nu'^{\pi'}_{[n] \setminus (R \cup R' \cup H_1)}(\mathcal{F}'''_1) \geq \nu'^{\pi'}_{[n] \setminus (R \cup H_1)}({\mathcal{F}_1''}^{\pi'}) \geq \nu'_{[n] \setminus (R \cup H_1)}(\mathcal{F}''_1) \geq \frac{1}{4}.$$
	For any $x'' \in \mathcal{F}'''_1, y'' \in \mathcal{F}'''_2$, there exist $x' \in \mathcal{F}'_1, y' \in \mathcal{F}''_2$ such that
    $$\agrparam{[n]\setminus (H_1 \cup H_2 \cup R)}{\pi'(x')}{\pi'(y')} = \agrparam{[n]\setminus (H_1 \cup H_2 \cup R \cup R')}{x''}{y''},$$
    and, therefore,
    $$\agrparam{[n]\setminus (H_1 \cup H_2 \cup R)}{x'}{y'} \leq \agrparam{[n]\setminus (H_1 \cup H_2 \cup R \cup R')}{x'}{y'}.$$
    Thus, it is sufficient to find $x'' \in \mathcal{F}'''_1, y'' \in \mathcal{F}'''_2$ with
    $$\agrparam{[n]\setminus (H_1 \cup H_2 \cup R \cup R')}{x''}{y''} = 0.$$

    \textbf{Step $5$.}
	By Claim~\ref{claim: averaging argument} there exist $z_i \in [m]^{(H_1 \cup H_2) \setminus (H_i \cup R \cup R')}$ such that
	$$\nu''_{[n]\setminus(R\cup R' \cup H_1 \cup H_2)}(\mathcal{F}'''_i((H_1 \cup H_2) \setminus (H_i \cup R \cup R') \rightarrow z_i)) \geq \nu''_{[n]\setminus(R\cup R' \cup H_i)}(\mathcal{F}'''_i).$$
	Let $\mathcal{F}''''_i = \mathcal{F}'''_i((H_1 \cup H_2) \setminus (H_i \cup R \cup R') \rightarrow z_i)$. Note, that only at this step we get the same set of coordinates for both families.

    \textbf{Step $6$.}
	Finally, we apply Lemma \ref{lemma: hoffman bound} for $\mathcal{F}''''_1$ and $\mathcal{F}''''_2$ to find $x'' \in \mathcal{F}''''_1$ and $y'' \in \mathcal{F}''''_2$ with $\agr(x'', y'') = 0$. Since $\nu'$ is $b^2$-balanced and $\pi' \in \Pi_{m', m'', b^2}^{\otimes [n] \setminus R}$, by Claim \ref{claim: balanced consistency} $\nu''$ is $b^4$-balanced. Thus, we get $\nu''_i(x) \leq \frac{b^4}{m''} \leq \frac{1}{3}$. Therefore, if $\mathcal{F}''''_1$ and $\mathcal{F}''''_2$ are cross-intersecting, we may apply Lemma \ref{lemma: hoffman bound} with $\lambda = \frac{1}{3}$ and get
	$$\frac{1}{8} \leq \nu''(\mathcal{F}''''_1)\nu''(\mathcal{F}''''_2) \leq \left(\frac{\lambda}{1-\lambda}\right)^2\left(1-\nu''(\mathcal{F}''''_1)\right)\left(1-\nu''(\mathcal{F}''''_2)\right) \leq \frac{3}{32}.$$
	However, $\frac{1}{8} > \frac{3}{32}$ and, therefore, there exist $x'' \in \mathcal{F}''''_1, y'' \in \mathcal{F}''''_2$ with $\agr(x'', y'') = 0$. Thus, there exist $(x_{R}, x_{R'}, z_1, x'') \in \mathcal{F}_1$ and $(y_{R}, y_{R'}, z_2, y'') \in \mathcal{F}_2$ such that $\pi_{R}(x_{R})=x, \pi_{R}(y_{R})=y, \pi'_{R'}(x_{R'})=x', \pi'_{R'}(y_{R'})=y'$ and $\agr(x, y)=\agr(x', y')=\agr(x'', y'')=0$. Therefore,
    \begin{align*}
    \agrparam{[n] \setminus (H_1 \cup H_2)}{(x_{R}, x_{R'}, z_1, x'')}{(y_{R}, y_{R'}, z_2, y'')} = 0. & \qedhere
    \end{align*}
\end{proof}

\section{Case $m\ge \poly(t)$, $n \ge \poly(t) \log m$}
\label{section: measure boosting argument case}

In this section, we prove the following theorem.

\begin{theorem}    
	\label{theorem: n > poly(t) ln m}
	Let $t \geq 2$,
    \begin{align}
        n \geq t + \frac{6t^2\ln m}{\ln \left(\frac{9}{8}\right)} + 2 + \frac{6t\ln m + 2 \ln \left(\frac{32}{3}\right)}{\ln 2}, \label{eq: bound on n, case n > poly(t)log(m)}
    \end{align}
    \begin{align}
        m \geq \frac{48}{\ln\left(\frac{9}{8}\right)}t^3\ln m + 16t, \label{eq: bound 1 on m, case n > poly(t)log(m)}
    \end{align}
    \begin{align}
        m \geq 2 \left(\frac{3}{2}\right)^{2}C_1(2)\left(\ln\left(\frac{32}{3}m^{3t}\right) \right)^{2C_2(2)}, \label{eq: bound 2 on m, case n > poly(t)log(m)}
    \end{align}
	where $C_1(b)$ and $C_2(b)$ are from Lemma \ref{lemma: cross matching for significant measure} and 
    \begin{align}
        m \geq 48\left(\frac{3}{\ln \frac{9}{8}}t^2 \ln m + 1\right). \label{eq: bound 3 on m, case n > poly(t)log(m)}
    \end{align}
	Let $\mathcal{F} \subset [m]^n$ be such family, that $\agr(x, y) \neq t - 1$ for any $x, y \in \mathcal{F}$ and $|\mathcal{F}| \geq m^{n-t}$. Then there exist $Z \subset n$ and $x \in [m]^{Z}$ such that $|Z|=t$ and $\mathcal{F} = [m]^{n}[Z \rightarrow x]$. 
\end{theorem}

The proof is divided into the following steps.

\textbf{Step $1$} Using Lemma \ref{lemma: spread approximation} we decompose $\mathcal{F}$ into homogenous families $\mathcal{F}[Z_i \rightarrow x_i]$, $i \in [l]$, with significant measures and a small residue $\mathcal{R}$.

\textbf{Step $2$.} Arguing indirectly, we prove that $\agrparam{Z_i \cap Z_j}{x_i}{x_j} \geq t$ for all $i, j$. First, assuming the contrary, we use Lemma \ref{lemma: homogenous avoiding} to obtain subfamilies $\mathcal{F}'_i \subset \mathcal{F}[Z_i \rightarrow x_i]$ and $\mathcal{F}'_j \subset \mathcal{F}[Z_j \rightarrow x_j]$ such that $\agrparam{Z_i \cup Z_j}{y_i}{y_j} = \agrparam{Z_i \cap Z_j}{x_i}{x_j}$ for all $y_i \in \mathcal{F}'_i, y_j \in \mathcal{F}'_j$. Second, we use Lemma \ref{lemma: large restriction with large probavility} to find a common restriction $H \rightarrow y$ of families $\mathcal{F}'_i$ and $\mathcal{F}'_j$ on a set of coordinates $H$ with size $|H| = (t-1) - \agrparam{Z_i \cap Z_j}{x_i}{x_j}$, which leaves the measures of the families significant. Finally, we use Lemma \ref{lemma: cross matching for significant measure} to find a cross-disagreement in families $\mathcal{F}'_i(Z_i \cup H \rightarrow (x_i, y))$ and $\mathcal{F}'_j(Z_j \cup H \rightarrow (x_j, y))$. Since this cross-disagreement implies existence of codes, agreeing exactly in $(t-1)$ coordinates, in $\mathcal{F}$, we conclude that $\agrparam{Z_i \cap Z_j}{x_i}{x_j} \geq t$ for all $i, j$.

\textbf{Step $3$.} We use Corollary \ref{corollary: simplification argument} to prove that there is $Z_0 \subset [n]$ and $x_0 \in [m]^{Z_0}$ such that $|Z_0| = t$ and for all $i \in [l]$ we have $Z_0 \subset Z_i$ and $x_0 = x_i|_{Z_0}$. That is, we prove that all the family $\mathcal{F}$, except the small residue $\mathcal{R}$, is contained in $[m]^{n}[Z_0 \rightarrow x_0]$ for some $Z_0$ with $|Z_0| = t$.

\textbf{Step $4$.} We prove that the residue $\mathcal{R}$ is empty, that is $\mathcal{F} \subset [m]^{n}[Z_0 \rightarrow x_0]$. In this step we deal with huge family $\mathcal{F}[Z_0 \rightarrow x_0]$ and small family $\mathcal{R}$, so we use Lemma \ref{lemma: unbalanced cross matching} to find cross-disagreement in the unbalanced case.

The following lemma is a key ingredient of the second step of the proof of Theorem \ref{theorem: n > poly(t) ln m}.

\begin{lemma}
	\label{lemma: big common restriction}
	Let $\nu$ be a $b$-balanced product measure on $[m]^{n}$. For $i \in \{1, 2\}$ let $Z_i \subset [n], x_i \in [m]^{Z_i}$ and $\mathcal{F}_i \in [m]^{[n] \setminus Z_i}$ are $\tau$-homogenous with respect to $\nu_{Z_i \rightarrow x_i}$. Let $\nu_{Z_i \rightarrow x_i}(\mathcal{F}_i) = \alpha_i$. Let $t, q$ be integer numbers such that $|Z_i| \leq q$ for $i \in \{1, 2\}$ and
    $\agrparam{Z_1 \cap Z_2}{x_1}{x_2} < t$.
    If
	$$\tau \leq \left(1 - \frac{2qb}{m}\right)\left(\frac{3}{2}\right)^{\frac{1}{t-1}}$$
	then there exist $H \subset [n] \setminus (Z_1 \cup Z_2)$ with
    $|H| = t - 1 - \agrparam{Z_1 \cap Z_2}{x_1}{x_2}$,
    $y \in [m]^{H}$ and $\mathcal{F}'_i \subset \mathcal{F}_i(H \rightarrow y)$ such that
	$$\nu_{Z_i \sqcup H \rightarrow (x_i, y)}\left( \mathcal{F}'_i \right) \geq \frac{1}{2}\left(1 - \frac{q\tau b}{m}\right)\alpha_i,$$
	$\mathcal{F}'_i$ are $\tau'$-homogenous with respect to $\nu_{Z_i \sqcup H \rightarrow (x_i, y)}$ with
	$$\tau' = 2\left(\frac{\tau}{\left(1 - \frac{q \tau b}{m}\right)}\right)^{t},$$
	and for any $z_i \in \mathcal{F}'_i$ we have
    $\agrparam{Z_1 \cup Z_2 \setminus Z_i}{z_i}{x_{3-i}} = 0$.
\end{lemma}

Note, that the last condition guarantees that for any $z_1 \in \mathcal{F}'_1, z_2 \in \mathcal{F}'_2$ we have
\[
\begin{split}
	&\agr((x_1, y, z_1), (x_2, y, z_2)) = \agrparam{Z_1 \cap Z_2}{x_1}{x_2} + |H| + \agrparam{[n] \setminus (Z_1 \cup Z_2 \cup H)}{z_1}{z_2} = \\&
	t - 1 + \agrparam{[n] \setminus (Z_1 \cup Z_2 \cup H)}{z_1}{z_2}.
\end{split}
\]

\begin{proof}
	Let $h = t - 1 - \agrparam{Z_1 \cap Z_2}{x_1}{x_2}$. Let $H \in \binom{[n] \setminus (Z_1 \cup Z_2)}{h}$ be arbitrary. 
	
	Let
    $\mathcal{F}''_i = \{y \in \mathcal{F}_i| \agrparam{Z_1 \cup Z_2 \setminus Z_i}{y}{x_{3-i}} = 0\}$.
    By Lemma \ref{lemma: homogenous avoiding}
	$$\nu_{Z_i \rightarrow x_i}(\mathcal{F}''_i) \geq \left(1 - \frac{q\tau b}{m}\right)\alpha_i$$
	and $\mathcal{F}''_i$ is $\tau''$-homogenous with $\tau''=\left(\frac{\tau}{\left(1 - \frac{q\tau b}{m}\right)}\right)$. Since $\tau < 2$, we get 
	$$\tau'' < \left(\frac{\tau}{\left(1 - \frac{2qb}{m}\right)}\right) \le \left(\frac{3}{2}\right)^{\frac{1}{t-1}}$$
	
	Let $\textbf{y} \sim \nu_H$ be a random element of $[m]^H$. By Lemma \ref{lemma: large restriction with large probavility} for $p=\frac{1}{2}$ we have
	$$\mathbb{P}\left(\nu_{[n]\setminus (Z_i \cup H)}\left(\mathcal{F}''_i(H \rightarrow \textbf{y})\right) \geq \frac{1}{2}\nu_{Z_i \rightarrow x_i}(\mathcal{F}''_i) \right) > \frac{1}{2}.$$
	Therefore, there exists $y \in [m]^H$ such that
	$$\nu_{[n]\setminus (Z_i \cup H)}\left(\mathcal{F}''_i(H \rightarrow y)\right) \geq \frac{1}{2}\nu_{Z_i \rightarrow x_i}(\mathcal{F}''_i) \geq \frac{1}{2}\left(1 - \frac{q\tau b}{m}\right)\alpha_i$$
	for both $i \in \{1, 2\}$. Let $\mathcal{F}'_i = \mathcal{F}''_i(H \rightarrow y).$ 
	
	We need to check that $\mathcal{F}'_i$ are $\tau'$-homogenous with respect to $\nu_{[n] \setminus (Z_i \cup H)}$ with
    $$\tau' = 2\left(\tau''\right)^{t} = 2\left(\frac{\tau}{\left(1 - \frac{q \tau b}{m}\right)}\right)^t.$$
	Indeed, for any $H' \subset [n] \setminus (Z_i \cup H)$ and $x' \in [m]^{H'}$, we get
	
		\[
	\begin{split}
		&\nu_{[n]\setminus (Z_i \cup H \cup H')}\left(\mathcal{F}'_i(H' \rightarrow x')\right) = \nu_{[n]\setminus (Z_i \cup H \cup H')}\left(\mathcal{F}''_i(H'\cup H \rightarrow (x', y))\right) \leq \\&
		\leq \tau''^{|H'| + |H|}\nu_{[n]\setminus Z_i}(\mathcal{F}''_i) \leq 2\tau''^{|H'| + |H|} \nu_{[n]\setminus (Z_i \cup H)}(\mathcal{F}'_i) \leq \\&
		\leq 2\tau''^{|H'| + t - 1} \nu_{[n]\setminus (Z_i \cup H)}(\mathcal{F}'_i) \leq \left(2\tau''^{t}\right)^{|H'|}\nu_{[n]\setminus (Z_i \cup H)}(\mathcal{F}'_i). \qedhere
	\end{split}
	\]
	
\end{proof}

The following lemma combines Lemmas \ref{lemma: cross matching for significant measure} and \ref{lemma: big common restriction} into the second step of the proof of Theorem \ref{theorem: n > poly(t) ln m}.

\begin{lemma}
	\label{lemma: t-intersectionness of approximation}
	Let $\mathcal{F} \subset [m]^n$ and $\nu$ be a $b$-balanced product measure on $[m]^n$ for $b \ge 2$. For $i \in \{1, 2\}$ let $Z_i \subset [n]$ and $x_i \in [m]^{Z_i}$. Let $t, q$ be integer numbers such that $|Z_i| \leq q$ for $i \in \{1, 2\}$ and
    $\agrparam{Z_1 \cap Z_2}{x_1}{x_2} < t$.
    Let families $\mathcal{F}_i \subset \mathcal{F}[Z_i \rightarrow x_i]$ be such that $\mathcal{F}(Z_i \rightarrow x_i)$ is $\tau$-homogenous with respect to $\nu_{Z_i \rightarrow x_i}$. Define $\alpha_i = \nu_{Z_i \rightarrow x_i}(\mathcal{F}_i(Z_i \rightarrow x_i))$.
    If
	$$\tau < \left(1 - \frac{2qb}{m}\right)\left(\frac{3}{2}\right)^{\frac{1}{t-1}},$$
	$$n > 2q + t + \frac{\ln(\left(\frac{1}{2}\left(1 - \frac{q\tau b}{m}\right)\alpha_1\right)^{-1}) + \ln(2\left(\frac{1}{2}\left(1 - \frac{q\tau b}{m}\right)\alpha_2\right)^{-1})}{\ln 2}$$
	and
	$$m > 2\left(\frac{3}{2}\right)^{\frac{t}{t-1}} C_1(b) \left(\ln\left(2\left(\left(1 - \frac{q\tau b}{m}\right)\alpha_1\right)^{-1}\right)\ln\left(4\left(\left(1 - \frac{q\tau b}{m}\right)\alpha_2\right)^{-1}\right)\right)^{C_2(b)},$$
	where $C_1(b)$ and $C_2(b)$ are from Lemma \ref{lemma: cross matching for significant measure}, then there exist $x'_i \in \mathcal{F}[Z_i \rightarrow x_i]$ such that $\agr(x'_1, x'_2) = t - 1$.
\end{lemma}

Note, that $(Z_1, x_1)$ and $(Z_2, x_2)$ are not required to be distinct.

\begin{proof}
	Let $H \subset [n]\setminus (Z_1 \cup Z_2)$ and $y \in [m]^{H}$ be from Lemma \ref{lemma: big common restriction}. Define $\mathcal{F}'_i = \mathcal{F}_i(H \cup Z_i \rightarrow (x_i, y))$. Put $\alpha'_i = \nu_{[n] \setminus (H \cup Z_i)}(\mathcal{F}'_i)$ and
	$$\tau' = 2\left(\frac{\tau}{\left(1 - \frac{q \tau b}{m}\right)}\right)^{t} < 2\left(\frac{3}{2}\right)^{\frac{t}{t-1}}.$$
	By Lemma \ref{lemma: big common restriction}, families $\mathcal{F}'_i$ are $\tau'$-homogenous with respect to $\nu_{[n] \setminus (H \cup Z_i)}$,
	$$\alpha_i' \geq \frac{1}{2}\left(1 - \frac{q\tau b}{m}\right)\alpha_i.$$
 
	Therefore, by Lemma \ref{lemma: cross matching for significant measure}, there exist $z_i \in \mathcal{F}'_i$ such that
    $\agrparam{[n] \setminus (H \cup Z_1 \cup Z_2)}{z_1}{z_2} = 0$.
    Let $x_i' = (x_i, y, z_i)$. By the last condition from Lemma \ref{lemma: big common restriction} we get
    $$\agrparam{Z_2 \setminus Z_1}{z_1}{x_2} = \agrparam{Z_1 \setminus Z_2}{x_1}{z_2} = 0.$$
    So we have $x_i' \in \mathcal{F}[Z_i \rightarrow x_i]$ and
    \[
    	\begin{split}
    		&\agr(x'_1, x'_2) = |H| + \agrparam{Z_1 \cap Z_2}{x_1}{x_2} + \agrparam{Z_1 \setminus Z_2}{x_1}{z_2} + \agrparam{Z_2 \setminus Z_1}{z_1}{x_2} + \\&
    		+ \agrparam{[n] \setminus (H \cup Z_1 \cup Z_2)}{z_1}{z_2} = (t-1-\agrparam{Z_1 \cap Z_2}{x_1}{x_2}) + \\&
    		+ \agrparam{Z_1 \cap Z_2}{z_1}{z_2} + 0 + 0 + 0 = t-1.
    	\end{split}
    \]
\end{proof}

The next lemma proves that if almost all family $\mathcal{F}$ is contained in $[m]^{n}[Z \rightarrow x]$ for some $Z$ with size $t$, then all $\mathcal{F}$ is contained in $[m]^{n}[Z \rightarrow x]$. We use it in the last step of the proof of Theorem \ref{theorem: n > poly(t) ln m}.

\begin{lemma}
    \label{lemma: almost all to all}
    Let $m \geq 8$. Let $\mathcal{F} \subset [m]^n$ be such a family that $\agr(x, y) \neq t - 1$ for any $x, y \in \mathcal{F}$ and $|\mathcal{F}| \geq m^{n-t}$. Assume that there exists $Z \subset n$ and $x \in [m]^{Z}$ such that $|Z| = t$ and $|\mathcal{F} \setminus [m]^n[Z \rightarrow x]| \leq m^{n - 3t}$. Then $\mathcal{F} = [m]^n[Z \rightarrow x]$.
\end{lemma}

\begin{proof}
    Let $\mathcal{F}_1 = \mathcal{F}[Z \rightarrow x]$ and $\mathcal{F}_2 = \mathcal{F} \setminus \mathcal{F}_1$. We need to prove $\mathcal{F} \subset [m]^{n}[Z \rightarrow x]$, that is,  $\mathcal{F}_2 = \emptyset$. Since $|\mathcal{F}| \geq m^{n-t} = |[m]^{n}[Z \rightarrow x]|$, it implies $\mathcal{F} = [m]^n[Z \rightarrow x]$. Let $\mathcal{\delta} = \mu(\mathcal{F}_2),$ where $\mu$ is the uniform measure of $[m]^{n}$. The condition $|\mathcal{F} \setminus [m]^n[Z \rightarrow x]| \leq m^{n - 3t}$ implies $\delta \leq m^{-3t}$.
    
    If $n < 3t$, the condition $|\mathcal{F} \setminus [m]^n[Z \rightarrow x]| \leq m^{n - 3t}$ immediately implies $|\mathcal{F}_2| < 1$ and, therefore, $\mathcal{F}_2 = \emptyset$. Thus, we may assume $n \geq 3t > 2t - 1$, that is, $[n] \setminus Z > t-1$. Let $Z' \in {[n] \setminus Z \choose t-1}$ be an arbitrary subset of $(t-1)$ elements of $[n] \setminus Z$.  By Claim \ref{claim: averaging argument} there exists $(y, y') \in [m]^{Z}\times [m]^{Z'}$ such that $\mu_{[n] \setminus (Z \cup Z')}(\mathcal{F}_2(Z\cup Z' \rightarrow (y, y'))) \geq \mu(\mathcal{F}_2) = \delta$. Let $\mathcal{F}'_2 = (\mathcal{F}_2(Z\cup Z' \rightarrow (y, y'))$.
    
    If $\mathcal{F}_2 \neq \emptyset$, we have $\mathcal{F}_2[Z\cup Z' \rightarrow (y, y')] \neq \emptyset$ and therefore $y \neq x$. Thus, we have $0 \leq \agr(x, y) \leq t - 1$, that is $0 \leq (t - 1) - \agr(x, y) \leq t - 1$. Since $m \geq 2$, it implies that there exists $x' \in [m]^{Z'}$ such that $\agr(x', y') = (t - 1) - \agr(x, y)$ and, therefore, $\agr((x, x'), (y, y')) = t-1$. Fix arbitrary such $x'$ and let $\mathcal{F}'_1 = \mathcal{F}_1(Z\cup Z' \rightarrow (x, x'))$. The assumption $\mathcal{F} \geq m^{n-t}$ implies $|[m]^n[Z \rightarrow x] \setminus \mathcal{F}[Z \rightarrow x]| = m^{n-t} - |\mathcal{F}_1| = m^{n-t} - |\mathcal{F}| + |\mathcal{F}_2| \leq |\mathcal{F}_2| = \delta m^{n}$. Since $[m]^n[Z \cup Z' \rightarrow (x, x')] \setminus \mathcal{F}[Z \cup Z' \rightarrow (x, x')] \subset [m]^n[Z \rightarrow x] \setminus \mathcal{F}[Z \rightarrow x]$, we get $|\mathcal{F}'_1| = m^{n-2t+1} - |[m]^n[Z \cup Z' \rightarrow (x, x')] \setminus \mathcal{F}[Z \cup Z' \rightarrow (x, x')]| \geq m^{n-2t+1} - \delta m^{n}$ and, therefore, $\mu_{[n] \setminus (Z \cup Z')}(\mathcal{F}'_1) \geq 1 - \delta m^{2t-1}.$

    We claim that there is no $x'' \in \mathcal{F}'_1, y'' \in \mathcal{F}'_2$ such that $\agr(x'', y'') = 0$. Indeed, otherwise, we have $(x, x', x'') \in \mathcal{F}, (y, y', y'') \in \mathcal{F}$ and $\agr((x, x', x''), (y, y', y'')) = \agr((x, x'), (y, y')) + \agr(x'', y'') = t - 1$. Thus, Lemma \ref{lemma: unbalanced cross matching} implies that $\mu_{[n] \setminus (Z \cup Z')}(\mathcal{F}'_1) + \mu_{[n] \setminus (Z \cup Z')}(\mathcal{F}'_2)^{\log_m(2)} \leq 1.$
    Substituting bounds $\mu_{[n] \setminus (Z \cup Z')}(\mathcal{F}'_1) \geq 1 - \delta m^{2t-1}$, $\mu_{[n] \setminus (Z \cup Z')}(\mathcal{F}')_2 \geq \delta$ and $m \geq 8$, we get $$1 - \delta m^{2t-1} + \delta^{1/3} \leq 1.$$
    Since $\delta \leq m^{-3t}$, we have
    $$1 \geq 1 - \frac{\delta \delta^{-2/3}}{m} + \delta^{1/3} = 1 + \left(1 - \frac{1}{m}\right)\delta^{1/3},$$
    and, consequently, $\delta = 0$. That is, we get $\mathcal{F}_2 = \emptyset$, and, therefore, $\mathcal{F} = [m]^{n}[Z \rightarrow x]$.
\end{proof}

\noindent \textit{Proof of Theorem \ref{theorem: n > poly(t) ln m}.}
\textbf{Step $1$.} Let $\{(Z_i, x_i)| i = 1 \ldots l \}$ be from Lemma \ref{lemma: spread approximation} for
	$$\tau = \left(\frac{9}{8}\right)^{\frac{1}{t}}$$ and
	$$q = \lfloor\log_{\tau}(m^{3t}) \rfloor + 1 = \left\lfloor \frac{3}{\ln \frac{9}{8}}t^2 \ln m \right\rfloor + 1.$$
	Note, that we have $\mu_{Z_i \rightarrow x_i}(\mathcal{F}_i(Z_i \rightarrow x_i)) \geq \tau^{-q} \geq \frac{m^{-3t}}{\tau} \geq \frac{1}{2}m^{-3t}$ for any $i \in [l]$ and $\mu(\mathcal{R}) \leq \tau^{-q} < m^{-3t}$, that is, $|\mathcal{R}| < m^{n-3t}$.
	
	\textbf{Step $2$.} Assume, for the sake of contradiction, that for some $i, j \in [l]$ we have
    $\agrparam{Z_i \cap Z_j}{x_i}{x_j} < t$.
    We apply Lemma \ref{lemma: t-intersectionness of approximation} for $(Z_i, x_i)$ and $(Z_j, x_j)$ with $\nu = \mu$, $b=2$, $\tau=\tau$ and $q=q$. Let us check its conditions. Since by requirement \eqref{eq: bound 1 on m, case n > poly(t)log(m)} we have
	$$\left(1 - \frac{4q}{m} \right)^{t} \geq 1 - \frac{4qt}{m} \geq 1 - \frac{4t\left( \log_{\tau}(m^{3t}) + 1\right)}{m} = 1 - \frac{\frac{12}{\ln\left(\frac{9}{8}\right)}t^3\ln m + 4t}{m} \geq \frac{3}{4},$$
	we get
	$$\tau = \left(\frac{3}{4}\right)^{\frac{1}{t}}\left(\frac{3}{2}\right)^{\frac{1}{t}} < \left(1-\frac{4q}{m}\right)\left(\frac{3}{2}\right)^{\frac{1}{t-1}}$$and therefore the assumption on $\tau$ from Lemma \ref{lemma: t-intersectionness of approximation} is satisfied.
	Since in terms of Lemma \ref{lemma: t-intersectionness of approximation} we have $(1 - \frac{q \tau b}{m}) \geq \frac{3}{4}$ and $\alpha_i \geq \frac{1}{2}m^{-3t}$, the right-hand side in the requirement on $m$ from this lemma does not exceed
	$$2 \left(\frac{3}{2}\right)^{2}C_1(2)\left(\ln\left(\frac{32}{3}m^{3t}\right) \right)^{2C_2(2)}$$
	and, therefore, the condition \eqref{eq: bound 2 on m, case n > poly(t)log(m)} implies that the assumption on $m$ from Lemma \ref{lemma: t-intersectionness of approximation} is satisfied.
	Finally, the right-hand side in the requirement on $n$ from Lemma \ref{lemma: t-intersectionness of approximation} does not exceed
	$$2q+t+\frac{2\ln\left(\frac{32}{3}m^{3t}\right)}{\ln 2} \leq t + \frac{6t^2\ln m}{\ln \left(\frac{9}{8}\right)} + 2 + \frac{6t\ln m + 2 \ln \left(\frac{32}{3}\right)}{\ln 2}$$
	and, therefore, this assumption is satisfied by \eqref{eq: bound on n, case n > poly(t)log(m)}.
	Thus, we may apply Lemma \ref{lemma: t-intersectionness of approximation} and get $x_1, x_2 \in \mathcal{F}$ with $\agr(x_1, x_2) = t-1$, a contradiction. Hence, for all $i, j \in [l]$ we have
    $\agrparam{Z_i \cap Z_j}{x_i}{x_j} \geq t$.
	
	\textbf{Step $3$.} Assume, again for the sake of contradiction, that there is no $Z \subset [n]$ and $x \in [m]^{Z}$ such that $|Z| = t$ and for all $i \in [l]$ we have $Z \subset Z_i$ and $x = x_i|_{Z}$. Requirement \eqref{eq: bound 3 on m, case n > poly(t)log(m)} implies that we may apply Corollary \ref{corollary: simplification argument} for $\varepsilon = \frac{1}{2}$. Thus, we get $$\left|\bigcup_{i \in [l]}\mathcal{F}[Z_i \rightarrow x_i]\right| \leq \left|\bigcup_{i \in [l]}[m]^n[Z_i \rightarrow x_i]\right| \leq \frac{1}{2} m^{n-t}$$
	and, therefore, $\mathcal{F} \leq \frac{1}{2}m^{n-t} + m^{n-3t}$. Since it contradicts the assumption $|\mathcal{F}| \geq m^{n-t}$, we conclude, that there exist $Z \subset [n]$ and $x \in [m]^{Z}$ such that $|Z| = t$ and for all $i \in [l]$ we have $Z \subset Z_i$ and $x = x_i|_{Z}$. Therefore, we have $\mathcal{F} \subset \mathcal{F}[Z \rightarrow x] \cup \mathcal{R}$ and
	$$|\mathcal{F} \setminus [m]^n[Z\rightarrow x]| \leq |\mathcal{R}| \leq m^{n-3t}.$$

    \textbf{Step $4$.} Finally, we apply Lemma \ref{lemma: almost all to all} and get $\mathcal{F} = [m]^{n}[Z \rightarrow x]$.
\qed

Although explicitly stated conditions \eqref{eq: bound 1 on m, case n > poly(t)log(m)}, \eqref{eq: bound 2 on m, case n > poly(t)log(m)}, \eqref{eq: bound 3 on m, case n > poly(t)log(m)}, \eqref{eq: bound on n, case n > poly(t)log(m)} in Theorem \ref{theorem: n > poly(t) ln m} are convenient to follow the proof, when we apply the theorem, it may be difficult to check four lengthy requirements. Thus, we conclude this section with an easier to use qualitative corollary of Theorem \ref{theorem: n > poly(t) ln m}.

\begin{corollary}
    \label{corollary: n > poly(t) ln m}
    There exist a polynomial $n_0(t)$ of degree $2$ and a polynomial $m_0(t)$ of degree at most $300$ such that the following holds.
    If $n \geq n_0(t)\ln(m)$, $m \geq m_0(t)$, and $\mathcal{F} \subset [m]^n$ is such a family that $\agr(x, y) \neq t - 1$ for any $x, y \in \mathcal{F}$ and $|\mathcal{F}| \geq m^{n-t}$, then there exist $Z \subset n$ and $x \in [m]^{Z}$ such that $|Z|=t$ and $\mathcal{F} = [m]^{n}[Z \rightarrow x]$.
\end{corollary}

\begin{proof}
    We need to check that conditions of Corollary \ref{corollary: n > poly(t) ln m} imply the conditions of Theorem \ref{theorem: n > poly(t) ln m}. The case $t=1$, that is, the case of intersecting family of codes, is considered in \cite{EKR_for_codes}. Thus, we may assume $t \geq 2$. Requirements \eqref{eq: bound 1 on m, case n > poly(t)log(m)}, \eqref{eq: bound 2 on m, case n > poly(t)log(m)}, \eqref{eq: bound 3 on m, case n > poly(t)log(m)} from Theorem \ref{theorem: n > poly(t) ln m} have a form $m \geq \text{poly}(t, \log m)$ and the polynomial from the requirement \eqref{eq: bound 2 on m, case n > poly(t)log(m)} has the highest degree $2C_2(2) < 300$ in $t$. Thus, there exists a polynomial $m_0(t)$ of degree at most $300$ such that $m \geq m_0(t)$ implies requirements \eqref{eq: bound 1 on m, case n > poly(t)log(m)}, \eqref{eq: bound 2 on m, case n > poly(t)log(m)}, \eqref{eq: bound 3 on m, case n > poly(t)log(m)}. Finally, for $m > e$ the condition \eqref{eq: bound on n, case n > poly(t)log(m)} is a consequence of a condition $n > n_0(t)\ln(m)$, where
    $$n_0(t) = t + \frac{6t^2}{\ln \left(\frac{9}{8}\right)} + 2 + \frac{4t + 2 \ln \left(\frac{32}{3}\right)}{\ln 2}$$
    is a polynomial of degree $2$.
\end{proof}

\section{Case $m \ge n^{\poly(t)}, n \ge \poly(t)$}
\label{section: classical extremal set theory case}

In this section, we once again treat $[m]^n$ as a subset of $[mn] \choose n$. 

We prove the following theorem:

\begin{theorem} \label{m_geq_nt}
	Let $1 \le t \le n+1, n \ge50t^2, m \geq (2^{22}n)^{4t}$.
	Let $\mathcal{F} \subset [m]^n$ be such a family, that $\agr(x, y) \neq t - 1$ for any $x, y \in \mathcal{F}$ and $|\mathcal{F}| \geq m^{n-t}$. Then $|\mathcal{F}| = m^{n-t}$, moreover, there exists $Z \subset n$ and $x \in [m]^{Z}$ such that $|Z|=t$ and $\mathcal{F}$ coincides with $[m]^n[Z\rightarrow x]$. 
\end{theorem}

To prove this theorem, we will reduce a maximal $(t-1)$-avoiding family $\ff$ to an $(\cS, 2, t)$-system. In Section \ref {subsection: get cT}, we will prove a lemma that uses this system to obtain some nice decomposition of $\ff$ up to a small remainder. In Section  \ref{subsection: properties of decomposition}, we will derive useful properties of this decomposition and special subfamilies of its parts. In Section \ref{subsection: bound sizes}, we will prove that, given such a decomposition, one may conclude that a large part of $\ff$ is contained in $[m]^n[T]$ for some $t$-element $T$. Lastly, in Section \ref{subsection: proof of m_geq_nt}, we combine the results of previous subsections to prove the main theorem of this section.

\subsection{Decomposition via an $(\cS, 2, t)$-system} \label{subsection: get cT}
We will need the following corollary of Lemma \ref {get_T}. Later we will apply this corollary to an $(\mathcal{S}, 2, t)$-system $\mathcal{U}$ obtained using Lemma \ref{get_sst}.  The particular bounds on $n, m, q$ could be explained as follows. We start the proof of the main theorem with constructing a $\tau$-homogeneous core consisting of sets of uniformity not greater than $q$. Since not all sets of the original family contain an element of the core, we have to deal with a remainder $\cR$. Our technique allows us to bound its measure by $\tau^{-q}$, and we aim to bound it as $|\cR| \le m^{n-2t}$, so we need $\tau^q > m^{2t}$.

In terms of Lemma \ref{get_T} the parameter $\lambda$\ shall not be too small to get good bounds on $|\widetilde{\mathcal{T}}|$ and $\sum_{S \in \mathcal{S}\setminus \mathcal{S}[\widetilde{\mathcal{T}}]}|\mathcal{U}_S|$ , i.e. $\lambda = m^{-const}$. For such a $\lambda$ to meet conditions of Lemma \ref{get_T} the parameter $\tau$ has to be equal to $\tau = m^{-const/t}$. In order to meet $\tau^q > m^{2t}$ we need $q$ to be a polynomial of degree $2$ depending on $t$. The fact that all sets of the core should have uniformity less than $n$ motivates that $n$ is also bounded by a quadratic polynomial depending on $t$. Finally, we take $m$ to be `big enough' of the form $n^{\poly(t)}$, and we did not try to optimize this parameter.

\begin{corollary}\label{corollary: get T}
    Let $m, n, t, q$ be such integers that $t \ge 1, n \ge 50t^2, m \ge (2^{22}n)^{4t}, q = 24t^2$. Let $\mathcal{U} \subset [m]^n$ be an $(\mathcal{S}, 2, t)$-system and $\{\mathcal{U}_S\}_{S \in \mathcal{S}}$ be the decomposition of $\mathcal{U}$ from the definition of an $(\mathcal{S},s, t)$-system satisfying \textit{(1 --- 3)} of Lemma \ref{get_sst} with $\alpha = (5n)^{-1}, \tau=m^{1/12t}$. Then there exists a family $\widetilde{\mathcal{T}} \subset \partial_t[m]^n$, such that  \begin{equation}\label{eq_T_hat}|\widetilde{\mathcal{T}}| \le m^{1/3}\ln m.\end{equation} Moreover,  \begin{equation}\label{eq_U_without_T}\sum_{S \in \mathcal{S}\setminus \mathcal{S}[\widetilde{\mathcal{T}}]}|\mathcal{U}_S| \le m^{n-t-1/3}.\end{equation}
\end{corollary}

\begin{proof}

We apply Lemma \ref{get_T} directly with $\lambda = m^{-1/6}$. We start with verifying the assumptions. First we note that, in fact, $m \ge (2^{22}n)^{4t} \ge(2^{22}\cdot 50t^2)^{4t}\ge 8 \cdot 24^3 \cdot t^6 = 8q^3$. Then we verify that the assumption of Lemma \ref{get_T} holds for $\lambda$ chosen before. Put $\tau' \le \tau = m^{1/12t}$ to be the minimal homogenity of $\ff_S, S \in \mathcal{S}$. Since \textit{(3)} of Lemma \ref{get_sst} is satisfied, we have $$|\partial_{t-1}(\mathcal{U}_S)| \ge \left(\frac{\tau'}{1 - 2\alpha n}\right)^{-(t-1)}|\partial_{t-1}([m]^n(S))|.$$

For all $(Z, x) = S \in \mathcal{S}$ we have $|Z| \le q$, and it is easy to see that $|\partial_{t-1}([m]^n(S))| \ge |\partial_{t-1}([m]^{n-q})| = {n-q \choose t-1}\cdot m^{t-1}$, where the last equality is due to $n - q \ge t - 1$. Also recall that $\tau' \le m^{1/12t}$, 
$\alpha = (5n)^{-1}$ and $|\partial_{t-1}([m]^n)| = {n\choose t-1}\cdot m^{t-1}$. Thus, we have $$ |\partial_{t-1}(\mathcal{U}_S)| \ge \left(\frac{3/5}{\tau'}\right)^{t - 1}\cdot {n-q \choose t-1}\cdot m^{t-1} \ge \left(\frac{3/5}{m^{1/12t}}\right)^{t - 1}\cdot {n-q \choose t-1}\cdot m^{t-1}.$$

It remains to show that $$\left(\frac{3/5}{m^{1/12t}}\right)^{t - 1}\cdot {n-q \choose t-1} \ge m^{-1/6} {n \choose t-1} $$ which is equivalent to $$ m^{1/6} \ge \left(\frac{m^{1/12t}}{3/5}\right)^{t-1}\cdot \frac{n\cdot \ldots \cdot (n - t + 2)}{(n - q) \cdot \ldots \cdot(n - q - t + 2)}.$$ For all $0 \le i \le t-2$ we have $(n-i)/(n-i-q) < 2$, because $n \ge 50t^2\ge t -2 + 48t^2 = t-2+2q$. Thus, it is enough to prove the following bound: $$m^{1/6} \ge \left(\frac{m^{1/12t}}{3/5}\right)^{t-1}\cdot 2^{t-1}.$$ It is equivalent to $$m \ge m^{(t-1)/2t} \cdot \left(\frac{2^6 \cdot 5^6}{3^6}\right)^{t-1},$$ which is true since $$m\ge m^{1/2}\cdot m^{1/2} \ge m^{(t-1)/2t}\left(\frac{2^6 \cdot 5^6}{3^6}\right)^{t-1}.$$


Therefore, Lemma \ref{get_T} is applicable. We obtain a set of restrictions $\widetilde{\mathcal{T}}$ such that for any $(Z, x) \in \widetilde{\mathcal{T}}$ we have $|Z| = t$. We have $$ |\widetilde{\mathcal{T}}| \le \frac1\lambda(1 +2\ln(\lambda|\partial_{\le q}[m]^n|)) = \frac{1 + 2\ln{\left(m^{-1/6} \cdot \sum_{p = 0}^q{n \choose p}m^p\right)}}{m^{-1/6}} \leq
$$ $$\leq  \frac{1 + 2\ln{\left(2m^{-1/6} \cdot n^q m^q\right)}}{m^{-1/6}} \le \frac{1 + 2q(\ln m + \ln n)}{m^{-1/6}} \le 5qm^{1/6}\ln m \leq m^{1/3}\ln m,$$ where the last inequality holds because $m \ge (2^{22}n)^{4t} \ge (2^{22}\cdot 50t^2)^{4t}\ge 5^6\cdot 24^6\cdot t^{12} = (5q)^6$. Also, Lemma \ref{get_T} implies $$\sum_{S \in \mathcal{S}\setminus \mathcal{S}[\hat{\mathcal{T}}]}|\mathcal{U}_S| \le \frac{32q^3}{\lambda m}\ln(\lambda|\partial_{\le q}[m]^ n|)\cdot m^{n-t} = \frac{32q^3}{m^{5/6}}\ln{\left(m^{-1/6} \cdot \sum_{p = 0}^q{n \choose p}m^p\right)}\cdot m^{n-t} \le$$ $$\le \frac{128q^4 \ln m}{m^{5/6}}\cdot m^{n-t}\le m^{n-t-1/3}.$$ The last inequality holds because $$m^{1/2} = m^{1/4}\cdot m^{1/4} \ge (2^{22}n)^t\cdot \frac{\ln m}{4} \ge 128q^4\ln m,$$ where we used $m^{1/4} = e^{1/4 \ln m} \ge 1 + \frac{1}{4} \ln m$,  $(2^{22}n)^t \ge 2^9(24t^2)^4$ and $$(2^{22}n)^t \ge (2^{22}\cdot 50t^2)^t \ge 2^{22t}t^8 \ge 2^9\cdot 24^4 \cdot t^8.$$
\end{proof}

\subsection{Properties of incomplete $t$-stars}  \label{subsection: properties of decomposition}

In the previous section, we constructed a set of restrictions $\widetilde\cT$ which `represents' well the original family. We now study the incomplete $t$-stars which are subsets of $\ff[T]$ for $T \in \widetilde\cT$. In this section we do not use properties of $\widetilde \cT$ so we prove all the claims for arbitrary $T \in \partial_t[m]^n$, however, we will apply them only to $T \in \widetilde\cT$.

The following definition is crucial in our proof.

\begin{definition}
  Let $\ff$ be a subfamily of $[m]^n$, $T \in \partial_{t}[m]^n, \ff_T\subset \ff[T]$. We define $\ff^*_T\subset\ff(T)$ as  
\begin{align*}
    \ff_T^* = \left \{ F \in \ff_T(T) \mid \forall X \in \binom{F}{\le t - 1} \text{ the covering number of } \ff_T(T\cup X) \text{ is at least }  n+1 \right \}.
\end{align*}
\end{definition}

We will later see that $\ff^*_T$ is usually of almost the same size as $\ff_T(T)$, however its properties allow us to find pairs of elements of $\ff$ with fixed agreement given their shadows with the same agreement. This idea will be illustrated in the proof of the following claim.

Recall that we treat $[m]^n$ as a subset of $[mn] \choose n$. Since the proof is based on techniques originally developed for sets we use the same notation. We now prove that if $\ff$ is $(t-1)$-avoiding then $\partial_{t - 1} \ff^*_{T}$ are disjoint for distinct $T$.

\begin{claim}\label{claim: different shadows}
Let $m, n, t$ be arbitrary positive integers, $\ff \subset [m]^n$ be $(t-1)$-avoiding. For distinct $T_1, T_2 \in \partial_t\ff$ and each $F_1 \in \ff_{T_1}^*, F_2 \in \ff^*_{T_2}$ we have $|F_1 \cap F_2| < t - 1$.

\end{claim}
\begin{proof}
Assume the contrary, there exist some distinct $T_1, T_2 \in \partial_t\ff$ and $F_1 \in \ff^*_{T_1}, F_2 \in \ff^*_{T_2}$ such that $|F_1 \cap F_2| \ge t - 1$. Put $T_1 = (Z_1, x_1), T_2 = (Z_2, x_2), F_1 = ([n] \setminus Z_1, y_1), F_2 = ([n]\setminus Z_2, y_2)$.

Since $T_1$ and $T_2$ are distinct and $|Z_1| = |Z_2| = t$, for $p = \agr_{Z_1 \cap Z_2}(x_1, x_2)$ we have $0 \le p \le t - 1$. Also, $\agr_{[n]\setminus (Z_1 \cup Z_2)}(y_1, y_2) \ge t-1$. Choose  a set $Z \subset [n]\setminus(Z_1 \cup Z_2)$ of $(t - 1 - p)$ coordinates on which $y_1$ and $y_2$ agree, put $x = y_1|_Z = y_2|_Z$. Since $F_1 \in \ff^*_{T_1}, |Z| \le t-1$, we can choose $F_1' \in \ff_{T_1}(T_1)$ 
such that it avoids $F_2 \sqcup T_2$ outside of $Z$. By analogy, we choose $F_2' \in \ff_{T_2}(T_2)$ such that it avoids $F_1' \sqcup T_1$ outside of $Z$. Both $F_1'\sqcup T_1$ and $F_2'\sqcup T_2$ belong to $\ff$, however $$|(F_1'\sqcup T_1) \cap (F_2'\sqcup T_2)| = |T_1 \cap T_2| + |Z| = t - 1,$$ which leads to contradiction. 
\end{proof}


Set $\delta_0 = e^{- n/t^2}$. Recall that our goal is to show that a maximal $(t-1)$-avoiding finally should be a complete $t$-star. In order to show that, we will prove that at most one of the restrictions from $\widetilde\cT$ covers a big part of the original family. Later, we will prove that in fact this restriction is the center of a $t$-star which coincides with the original family. We are ready to prove the following proposition.
\begin{proposition}
Let $m, n, t$ be such positive integers that $n \ge 50t^2$. Let $\ff \subset [m]^n$ be $(t-1)$-avoiding. Then there are no two distinct $T_1, T_2 \in \partial_t\ff$ such that $\mu^{-T_1}(\ff^*_{T_1}) \ge \delta_0$ and $\mu^{-T_2}(\ff^*_{T_2}) \ge \delta_0$.
\end{proposition}

\begin{proof}

Assume the contrary. Put $T_1 = (Z_1, x_1), T_2 = (Z_2, x_2)$. By averaging (Claim \ref{claim: averaging argument}), we find restrictions $y_1 \in [m]^{Z_2\setminus Z_1}$ and $y_2 \in [m]^{Z_1 \setminus Z_2}$ such that $$\mu_{[n]\setminus Z_1\cup Z_2}(\ff^*_{T_1}(Z_2\setminus Z_1 \to y_1)) \ge \delta_0; \qquad \mu_{[n]\setminus Z_1\cup Z_2}(\ff^*_{T_2}(Z_1\setminus Z_2 \to y_2)) \ge \delta_0.$$ Put $\cG_1 = \ff_{T_1}^*(Z_2 \setminus Z_1 \to y_1), \cG_2 = \ff_{T_2}^*(Z_1\setminus Z_2 \to y_2)$. By Theorem \ref{KK_direct}, we have 
$$|\partial_{t-1}\cG_1| \ge\delta_0^{\frac{t-1}{n - |Z_2 \cup Z_1|}}\cdot|\partial_{t-1}m^{[n]\setminus Z_1\cup Z_2}| = $$ $$ =e^{-\frac{n(t-1)}{t^2(n - |Z_2 \cup Z_1|)}}\cdot |\partial_{t-1}m^{[n]\setminus Z_1\cup Z_2}| > $$ $$> \frac12{n - |Z_2 \cup Z_1| \choose t-1}m^{t - 1}.$$ The last inequality holds trivially for $t=1$. For $t \ge 2$ we have $n/t(n - 2t) \le 0.6 < \ln(2)$ and therefore
$$e^{-\frac{n(t-1)}{t^2(n - |Z_2 \cup Z_1|)}} \ge e^{-\frac{n}{t(n - 2t)}} > e^{-\ln2 }\ge \frac12.$$

Analogously, $|\partial_{t-1}\cG_2| \ge \frac12{n - |Z_1 \cup Z_2| \choose t-1}m^{n - |Z_1 \cup Z_2|}.$ Since $\partial_{t-1}\cG_1, \partial_{t-1}\cG_2 \subset \partial_{t-1}[m]^{n\setminus (Z_1\cup Z_2)}$, the intersection of $\partial_{t-1}\cG_1$ and $\partial_{t-1}\cG_2$ is not empty which contradicts Claim \ref{claim: different shadows}.
\end{proof}
\subsection{Geometrically scaling decomposition}  \label{subsection: bound sizes}

In this section, we discuss a $(t-1)$-avoiding family $\ff \subset [m]^n$ decomposed as $\cR \sqcup \bigsqcup_{T\in \widetilde{\cT}}\ff_T,$ where $\widetilde \cT \subset \partial_{t}\ff$, $\ff_T \subset \ff[T]$ for all $T \in \widetilde\cT$ and $\cR$ is a relatively small remainder. We will show that there is a restriction $T \in \widetilde\cT$ such that almost all sets of $\ff$ are contained in $\ff_T \vee\{T\}.$ We begin with bounding the number of parts of the decomposition depending on their measure.


\begin{claim}\label{claim: estimate t_i}
    Let $m, n, t$ be positive integers such that $n \ge 50t^2$. Let $\ff = \bigsqcup_{T\in \widetilde{\cT}}\ff_T$ be a $(t-1)$-avoiding subfamily of $[m]^n$, $\widetilde\cT \in \partial_t\ff$ and $\ff_T \subset \ff[T]$ for all $T \subset \widetilde\cT$. Set $\delta_0 = e^{-n/t^2}$, for $i \in \mathbb{N}$ put $\delta_i = \delta _0^{i+1}.$  Define \begin{align*}
	\widetilde{\cT}_i = \{T \in \widetilde{\cT} \mid \mu^{-T}(\ff^*_T) \in (\delta_{i}, \delta_{i - 1}]\}.
\end{align*} Then for any $\widetilde\cT_i$ we have $|\widetilde{\cT}_i| \le 2\delta_i^{-(t-1)/(n-t)}$.
\end{claim}
\begin{proof}
    Assume the contrary. From Theorem \ref{KK_direct} for all $T \in \widetilde{\cT}_i$ we have $$|\partial_{t-1}\ff^*_{T}| \ge {n - t\choose t-1}m^{t-1}\cdot\delta_i^{\frac{t-1}{n-t}}.$$ Claim \ref{claim: different shadows} guarantees that the $(t-1)$-shadows of $\ff^*_T$ are disjoint for different $T$. Thus the following series of inequality holds. $$\Bigg|\partial_{t-1}\bigcup_{T \in \widetilde{\cT}}\ff^*_{T}\Bigg| = \sum_{T \in \widetilde{T}}|\partial_{t-1}\ff^*_T| \ge $$$$\ge |\widetilde{\cT_i}|\cdot {n - t\choose t-1}m^{t-1}\cdot\delta_i^{\frac{t-1}{n-t}} \ge 2\cdot {n - t\choose t-1}m^{t-1} > {n \choose t-1}m^{t-1} = |\partial_{t-1}[m]^n|.$$ This leads us to contradiction. Note that $2{n - t\choose t-1} > {n \choose t-1}$ because $$\frac{{n \choose t-1}}{{n-t\choose t-1}} = \frac{n(n-1)\cdot\ldots\cdot(n-t+2)}{(n-t)\cdot\ldots\cdot(n - 2t + 2)} \le \left(1 + \frac{t}{n-2t+2}\right)^{t-1} < 2,$$ where the last inequality is due to $n \ge 50t^2 \Longrightarrow n - 2t + 2 \ge 2t^2$.
\end{proof}

Using Claim \ref{claim: estimate t_i}, we will show in two steps that there exists a set $T_0 \in \widetilde\cT$ such that $|\ff\setminus \ff_{T_0}^*|$ is relatively small. First, in Claim \ref{claim: sum of f*},  we prove that the total size of families $\ff_T^*$ except some $\ff^*_{T_0}$ is at most $2\sqrt{\delta_0}m^{n-t}.$ Then, in Claim \ref{claim: almost all in *}, we prove that $|\ff \setminus\bigcup_{T \in \widetilde\cT}\ff^*_T| \le m^{n-t-1/3}.$

\begin{claim}\label{claim: sum of f*}
    In terms of Claim \ref{claim: estimate t_i} set $T_0 \in \widetilde\cT$ be such that $\mu^{-T_0}(\ff_{T_0}^*)$ is the biggest among $\mu^{-T}(\ff_T^*), T\in \widetilde\cT$. Then $$\sum_{T \in \widetilde{\cT}}|\ff_{T}^*| \le |\ff^*_{T_0}| + 2 m^{n - t} \sqrt{\delta_0}.$$
\end{claim}
\begin{proof}
Fix any $T_0$ for which $\mu(\ff_{T_0})$ attains the maximum. According to Claim \ref{claim: estimate t_i}, we have
$$\sum_{T \in \widetilde{\cT}}|\ff_{T}^*| = |\ff^*_{T_0}| + \sum_{i = 1}^{\infty} \sum_{T \in \widetilde{\cT}_i} |\ff^*_{T_i}| \le |\ff^*_{T_0}| + 2 m^{n - t} \sum_{i = 1}^{\infty} \delta_{i}^{-(t-1)/(n-t)} \delta_{i - 1} = $$ $$= |\ff^*_{T_0}| + 2 m^{n - t} \sum_{i=1}^{\infty}\delta_0^{i - \frac{(i+1)(t-1)}{n-t}}.$$ For any $i\in \mathbb{N}$, since $n - t > 6(t-1)$, the following inequalities hold: $$i - \frac{(i+1)(t-1)}{n-t}\ge i - \frac{i+1}6 \ge \frac{2i}3$$ and we can proceed estimating $\sum_{T \in \widetilde{\cT}}|\ff_{T}^*|$ as $$\sum_{T \in \widetilde{\cT}}|\ff_{T}^*| \le |\ff^*_{T_0}| + 2 m^{n - t}\sum_{i=1}^{\infty}\delta_0^{2i/3} = |\ff^*_{T_0}| + 2 m^{n - t}\delta_0^{2/3}\cdot\frac1{1 - \delta_0^{2/3}} \le |\ff^*_{T_0}| + 2 m^{n - t} \sqrt{\delta_0}$$ The last inequality holds, because $\delta_0 = e^{-n/t^2} < e^{-3} < 0.1$.
\end{proof}

We will now bound $\sum_{T \in \widetilde{\cT}}|\ff_T(T) \setminus\ff^*_T|$.

\begin{claim}\label{claim: almost all in *}
    Let $m, n, t$ be positive integers such that $m \ge (2^{22}n)^{4t}, n \ge 50t^2$. Let $\ff = \bigsqcup_{T\in \widetilde{\cT}}\ff_T$ be a $(t-1)$-avoiding subfamily of $[m]^n$, $\widetilde\cT \in \partial_t\ff$, $\ff_T \subset \ff[T]$ for all $T \in \widetilde\cT$ and $|\widetilde\cT| \le m^{1/3}\ln m$. Then, we have
     $$\sum_{T \in \widetilde{\cT}}|\ff_T(T) \setminus\ff^*_T| \le m^{n-t-1/3}.$$
\end{claim}
\begin{proof}
Recall the definition of $\ff_T^*$. The sets from $\ff_T(T) \setminus \ff_T^*$ can be described as follows: there is a restriction $(Z, x), Z \in [n], x\in [m]^Z, |Z| \le t-1$ such that the covering number of $\ff_T(Z\to x)$ is less than $n+1$. It also means that the size of $\ff_T(T)(Z\to x)$ is at most $(n+1)\cdot m^{n-t-|Z|-1}$, and we can write $$\sum_{T \in \widetilde{\cT}}|\ff_T(T) \setminus\ff^*_T| \le |\widetilde\cT|\cdot\sum_{\substack{Z \in [n], |Z|\le t-1 \\ x\in [m]^Z}}(n+1)\cdot m^{n - t-|Z| - 1} \le 2n|\widetilde{\cT}|\cdot\sum_{Z \in [n], |Z|\le t-1} m^{n-t-1} \le $$ \begin{equation}\label{eq_all_T_but*}\le 2n \cdot m^{1/3}\ln m\cdot\sum_{p=0}^{t-1}{n\choose p}\cdot m^{n-t-1} \le 4n^t\ln m\cdot m^{n-t-2/3} \le m^{n-t-1/3},\end{equation} 
    where we used $$m \ge (2^{22}n)^{3t}\cdot (m^{1/12})^3  \ge(2^{22}n)^{3t}\cdot \left(\frac{\ln m}{12}\right)^3\ge4^3n^{3t} \cdot \ln^3 m = (4n^t\ln m)^3$$ in the last inequality.    
\end{proof}

We are now ready to prove that for some $T_0$ all but a constant factor sets from $\ff^*_{T_0}\vee\{T_0\}$ are contained in $\ff$, i.e. $\ff$ contains a large incomplete $t$-star.

\begin{claim}\label{claim: get first delta}
    Let $m, n, t$ be positive integers such that $m \ge (2^{22}n)^{4t}, n \ge 50t^2$. Suppose that a $(t-1)$-avoiding family $\ff \subset [m]^n, |\ff| \ge m^{n-t}$, can be decomposed into families $\cR \sqcup \bigsqcup_{T\in \widetilde{\cT}}\ff_T$ such that $\widetilde\cT \subset \partial_t\ff$, $\ff_T \subset \ff[T]$ for all $T \in \widetilde\cT$ and $|\cR| \le 2m^{n-t-1/3}, |\widetilde\cT| \le m^{1/3}\ln m$. Then, there are $T_0\in \widetilde\cT$ and $\delta < 1/e$ such that $$\big|\ff^*_{T_0}\vee\{T_0\}\big| = (1-\delta)m^{n-t}.$$
\end{claim}
\begin{proof}
The family $\ff$ can be partitioned into \begin{equation}\label{equation: decomposition of F}\ff =\cR\sqcup \left(\bigsqcup_{T \in \widetilde{\cT}}\ff_T \setminus (\ff^*_T \vee \{T\})\right) \sqcup\left(\bigsqcup_{T\in\widetilde\cT}(\ff^*_T \vee \{T\})\right).\end{equation} According to Claim \ref{claim: almost all in *}, we have $$\Bigg|\cR\sqcup \left(\bigsqcup_{T \in \widetilde{\cT}}\ff_T \setminus (\ff^*_T \vee \{T\})\right)\Bigg| \le |\cR| + \Big|\bigcup_{T \in \widetilde{\cT}}\ff_T(T) \setminus \ff^*_T\Big| \le 3m^{n-t-1/3}.$$

Claim \ref{claim: sum of f*} implies $$|\ff| \le 3m^{n - t - 1/3} + |\ff_{T_0}^*| + 2m^{n-t}\sqrt{\delta_0} = |\ff_{T_0}^*| + (2\sqrt{\delta_0} + 3m^{-1/3})m^{n-t},$$ where $\delta_0 = e^{-n/t^2}$ and $T_0 \in \widetilde\cT$ is such that $\mu^{-T_0}(\ff_{T_0}^*)$ is the biggest among $\mu^{-T}(\ff_T^*), T\in \widetilde\cT$. Since $|\ff| \ge m^{n-t}$, we have $$|\ff^*_{T_0}| \ge (1 - 2 \sqrt{\delta_0} - 3m^{-1/3}) m^{n - t}.$$ 

Set $\delta > 0$ such that $\big|\ff^*_{T_0}\vee\{T_0\}\big| = (1-\delta)m^{n-t}$. According to $m \ge (2^{22}n)^{4t} \ge 2^{22}$, we have $3m^{-1/3} \le 1/2e$, which implies $$\delta \le2\sqrt{\delta_0} + 3m^{-1/3} \le e^{-n/2t^2 +1} + 1/2e \le 1/e^2 +1/2e \le 1/e. $$ The second to last inequality is due to $-n/2t^2+1\le -2$.
\end{proof}

The bound $\delta \le 1/e$ is clearly not sufficient for our purposes, but it can be easily upgraded using the Kruskal--Katona theorem, as the following claim shows.

\begin{claim}\label{claim: sum less that delta/2}
In terms and conditions of Claim \ref{claim: get first delta} let  $T_0\in \widetilde\cT$ and $\delta < 1/e$ be such that $$\big|\ff \setminus(\ff^*_{T_0}\vee\{T_0\})\big| = (1-\delta)m^{n-t}.$$ Then the sum of sizes of $\ff_T^*, T\in \widetilde{T}\setminus\{T_0\}$, can be bounded as $$\sum_{T\in \widetilde{T}\setminus\{T_0\}}|\ff_T^*| \le \delta/2\cdot m^{n-t}$$
\end{claim}
\begin{proof}
Let $(Z_0, x_0)$ be $T_0$ written in the restriction form. By averaging (see Claim~\ref{claim: averaging argument}), for each $T = (Z_T, x_T) \in \widetilde\cT\setminus \{T_0\}$, find a restriction $y_T \in [m]^{Z_0\setminus Z_T}$ such that for $\cG_T = \ff^*_T(Z_0\setminus Z_T\to y_T)$ we have $\mu_{[n]\setminus Z_0\cup Z_T}(\cG_T) \ge \mu^{-T}(\ff^*_T).$ Theorem \ref{KK_direct} implies $$|\partial_{t-1}\ff_{T_0}^*| \ge  (1 - \delta)^{\frac{t-1}{n-t}}|\partial_{t-1}[m]^{[n]\setminus Z_0}|$$

Families $\partial_{t-1}\cG_T \subset \partial_{t-1}[m]^{[n]\setminus T_0\cup T}$ are disjoint for different $T \subset \widetilde\cT\setminus\{T_0\}$ and also disjoint from $\partial_{t-1}\ff^*_{T_0} \in \partial_{t-1}[m]^{[n]\setminus T_0}$ by Claim \ref{claim: different shadows}. Therefore the following inequality holds: $$\sum_{T \in \widetilde\cT \setminus \{T_0\}}|\partial_{t-1}\cG_T| \le \left(1 - (1-\delta)^{\frac{t-1}{n-t}}\right)\cdot|\partial_{t-1}[m]^{[n]\setminus Z_0}|\le$$  \begin{equation}\label{eq_sum_G_measures}
    \le \frac{t-1}{n-t}\cdot\delta|\partial_{t-1}[m]^{[n]\setminus Z_0}| < \delta|\partial_{t-1}[m]^{[n]\setminus Z_0}|.\end{equation}

We now use a contraposition of Theorem \ref{KK_direct}, i.e. Corollary \ref{KK_contra}, to bound the measure of a family via the size of its shadow. For any $T \in \widetilde\cT\setminus\{T_0\}$, we have 
\begin{align*}
	\mu^{- (Z \cup Z_0)}(\cG_T) \le \left ( \frac{|\partial_{t - 1} \cG_T|}{|\partial_{t - 1} [m]^{[n] \setminus (Z_T\cup Z_0)}|} \right )^{(n - t-|Z_0\setminus Z_T|)/(t-1)} \le \\ \le \left ( \frac{|\partial_{t - 1}\cG_T|}{|\partial_{t - 1} [m]^{[n] \setminus Z_0}| \cdot {n - t - |Z_0\setminus Z_T| \choose t-1} / {n - t\choose t-1}}\right )^{(n - 2t)/(t-1)}.
\end{align*}

Previously, we chose the restrictions $y_T$ in such a way that the measure of the remaining part was at least the measure of the original family. It allows us to bound the sum of measures of $\ff_T^*$ as follows: 
\begin{align}
    & \sum_{T \in \widetilde\cT\setminus\{T_0\}}\mu_{[n]\setminus Z_T}(\ff^*_T) \le \sum_{T \in \widetilde\cT\setminus\{T_0\}}\mu_{[n]\setminus Z_T\cup Z_0}(\cG_T)  \nonumber \\
    & \qquad \le \max_{T \in \widetilde\cT\setminus\{T_0\}} \left(\frac{{n-t\choose t-1}}{{n-t-|Z_0\setminus Z_T| \choose t-1}}\right)^{\frac{n-2t}{t-1}}\cdot \sum_{T \in \widetilde\cT\setminus\{T_0\}}\left ( \frac{|\partial_{t - 1}\cG_i|}{|\partial_{t - 1} [m]^{[n] \setminus Z_0}| }\right )^{\frac{n-2t}{t-1}}. \label{eq: shadow bound}
\end{align}

We analyze these two factors separately. We have
\begin{align}
    & \max_{T \in \widetilde\cT\setminus\{T_0\}}\left(\frac{{n-t\choose t-1}}{{n-t-|Z_0\setminus Z_T| \choose t-1}}\right)^{\frac{n-2t}{t-1}} \nonumber \\
    & \qquad \le \max_{T \in \widetilde\cT\setminus\{T_0\}}\left(\frac{(n-t)(n-t+1)\cdot\ldots\cdot(n-2t+2)}{(n-t-|Z_0\setminus Z_T|)\cdot\ldots\cdot(n - 2t -|Z_0\setminus Z_T| + 2)}\right)^{\frac{n-2t}{t-1}} \nonumber \\
    & \qquad \le \max_{T \in \widetilde\cT\setminus\{T_0\}}\left(1 + \frac{|Z_0\setminus Z_T|}{n - 2t - |Z_0\setminus Z_T| + 2}\right)^{(t-1)\frac{n-2t}{t-1}} \le \left( 1 + \frac{t}{n - 3t}\right)^{n-2t} \nonumber \\
    & \qquad \le \left( 1 + \frac{2t}{n}\right)^{n} \le e^{2t}. \label{eq_factor1}
\end{align}


To analyze the second factor, note that the function $x \mapsto x^{\frac{n-2t}{t-1}}$ is convex, thus, $(x_1, \ldots, x_k) \to \sum_{i=1}^k{x_i^{\frac{n-2t}{t-1}}}$ is also convex. Put $$u_T = |\partial_{t-1}\cG_T| / |\partial_{t-1}[m]^{[n]\setminus Z_0}|.$$ Inequality \eqref{eq_sum_G_measures} implies $\sum_{T \in \widetilde\cT \setminus \{T_0\}} u_T  < \delta$, thus, we can bound the second factor of~\eqref{eq: shadow bound} by a convex program
\begin{align*}
    & \max_{u_T, T \in \widetilde{\cT} \setminus \{T_0\}} \sum_{T \in \widetilde{T} \setminus \{T_0\}} u_T^{\frac{n - 2t}{t - 1}} \\
    & \qquad \text{subject to} \quad \begin{cases}
        u_T \ge 0, \\
        \sum_{T \in \widetilde{T} \setminus \{T_0\}} u_T \le \delta.
    \end{cases}
\end{align*}
A convex function attains its maximum over a simplex at one of its vertices, which implies \begin{equation}\label{eq_factor2}
    \sum_{T \in \widetilde\cT\setminus\{T_0\}}\left ( \frac{|\partial_{t - 1}\cG_i|}{|\partial_{t - 1} [m]^{[n] \setminus Z_0}| }\right )^{\frac{n-2t}{t-1}} \le  \left ( \frac{\delta |\partial_{t - 1} [m]^{[n] \setminus Z_0}|}{|\partial_{t - 1} [m]^{[n] \setminus Z_0}| }\right )^{\frac{n-2t}{t-1}}\le \delta^{\frac{n-2t}{t-1}}.\end{equation}

Combining \eqref{eq: shadow bound}, \eqref{eq_factor1}, \eqref{eq_factor2} and the fact that $\delta < 1/e$ we obtain 
$$\sum_{T\in \widetilde{\cT}\setminus{T_0}}|\ff_T^*| \le m^{n-t} \cdot \sum_{T \in \widetilde\cT\setminus\{T_0\}}\mu_{[n]\setminus Z_T}(\ff^*_T) \le m^{n-t}\cdot e^{2t}\cdot\delta^{\frac{n-2t}{t-1}}\le$$ $$\le m^{n-t}\cdot\delta^{\frac{n-2t}{t-1}-2t} \le m^{n-t}\delta^2 \le \delta/2 \cdot m^{n-t}.$$ The second to last inequality holds because $\frac{n-2t}{t-1} - 2t \ge 2$.    
\end{proof}
Now we can get a better bound on $\delta$ using the previous claim.

\begin{claim}\label{claim: T0 is almost everything}
    In terms and conditions of Claim \ref{claim: get first delta},  there are $T_0\in \widetilde\cT$ and $\delta < 6m^{-1/3}$ such that $$\big|\ff^*_{T_0}\vee\{T_0\}\big| = (1-\delta)m^{n-t}.$$
\end{claim}
\begin{proof}
According to Claim \ref{claim: get first delta} and Claim \ref{claim: sum less that delta/2},
there are $T_0 \in \widetilde\cT$ and $\delta \le 1/e$ such that $$\big|\ff^*_{T_0}\vee\{T_0\}\big| = (1-\delta)m^{n-t}$$ and $$\sum_{T\in \widetilde{\cT}\setminus\{T_0\}}|\ff_T^*| \le \delta/2\cdot m^{n-t}.$$ 
The decomposition \eqref{equation: decomposition of F} together with $|\ff| \ge m^{n-t}$ implies $$ \delta/2\cdot m^{n-t}\le \Bigg|\cR\sqcup \left(\bigsqcup_{T \in \widetilde{\cT} }\ff_T \setminus (\ff^*_T \vee \{T\})\right)\Bigg|\overset{\text{Claim } \ref{claim: almost all in *}}\le 3m^{n-t-1/3}.$$ Therefore $\delta < 6m^{-1/3}$.
\end{proof}

\subsection{Proof of Theorem \ref{m_geq_nt}}  \label{subsection: proof of m_geq_nt}

\begin{proof}[Proof of Theorem \ref{m_geq_nt}]
Let $\mathcal{F} \subset [m]^n$ be the maximal $(t-1)$-avoiding family, $|\mathcal{F}| \ge m^{n-t}$. We apply Lemma \ref{lemma: spread approximation} to $\mathcal{F}$ with $\tau = m^{1/12t}, q = 24t^2$ and obtain a decomposition $$\mathcal{F} = \mathcal{R} \cup \bigcup_{i\in [l]} \mathcal{F}_i,$$ such that for all $i \in [l]$ there are some $Z_i \subset [n], |Z_i| \le q$ and $x_i \in [m]^{Z_i}$ for which we have $\mathcal{F}_i \subset \mathcal{F}[Z_i \rightarrow x_i]$. Besides, $\mu(\mathcal{R}) \leq \tau^{-q}$, and, for all $i$, we have $\mu_{Z_i \rightarrow x_i}(\mathcal{F}_i(Z_i \rightarrow x_i)) \geq \tau^{-q}$ and $\mathcal{F}_i(Z_i \rightarrow x_i)$ is $\tau$-homogenous with respect to $\mu_{Z_i \rightarrow x_i}$

Put $\ff' = \bigcup_{i \in [l]}\ff_i$, $\mathcal{S} = \{(Z_1,x_1), \ldots, (Z_l, x_l)\}$. Now we apply Lemma \ref{get_sst} to $\ff'$ playing the role of $\ff$ and $\mathcal{S}$ being its $\tau$-homogenous core. We need to verify that all quantitative conditions hold. First, $t \le n+1$ and $n \ge q + t -1 = 24t^2 + t -1$ hold from the statement of Theorem \ref{m_geq_nt}. Second, $q = 24t^2 \ge t$ because $t \ge 1$. Finally, for $\alpha = (5n)^{-1} \in [0, (4n)^{-1})$ we have
$$m \ge m^{1/12} \cdot (2^{22}n)^t \ge \max\{2^{15}\lceil\log_2n\rceil, 48t^2\} \cdot 2 \cdot (5n)^{t-1} \cdot m^{1/12} = $$$$ = \max\{2^{15}\lceil\log_2n\rceil, 2q\}\cdot 2\alpha^{1-t}\tau^t.$$ The second inequality is due to $2^{22}n \ge 2\cdot 2^{15}\lceil\log_2n\rceil$ and $2^{22}n \ge 2^{22}\cdot50t^2\ge 2\cdot 48t^2$.

Thus, we obtain families $\mathcal{U}_1, \ldots, \mathcal{U}_l$ satisfying \textit{(1 --- 3)} of Lemma \ref{get_sst}. It is easy to see that $\mathcal{U} = \bigcup_{i \in [l]}\mathcal{U}_i$ is an $(\mathcal{S}, 2, t)$-system. We now apply Corollary \ref{corollary: get T} to $\mathcal{U}$ and obtain a set of restrictions $\widetilde{\cT}$, satisfying \eqref{eq_T_hat} and \eqref{eq_U_without_T}. We can decompose $\ff$ as $$\ff = \cR \cup \ff[\widetilde{\cT}] \cup \bigcup_{S \in \cS \setminus \cS[\widetilde{\cT}]} \ff_S.$$ Then we have $$|\ff\setminus \ff[\widetilde{\cT}]| \le |\cR| + \sum_{S \in \cS \setminus \cS[\widetilde{\cT}]}|\ff_S| \le m^n\cdot \tau^{-q} + (1 - 2\alpha n)^{-1}\sum_{S \in \cS \setminus \cS[\widetilde{\cT}]}|\cU_S| \le $$ \begin{equation}\label{eq_F-FT}\overset{\eqref{eq_U_without_T}}\le m^n \cdot m^{-2t} + \frac53\cdot m^{n-t-1/3}\le 2m^{n-t-1/3}.\end{equation}

We partition $\ff[\widetilde{\cT}]$ into families $\ff_T \subset \ff[Z_T \to x_T]$ for all $(Z_T, x_T) = T \in \widetilde{\cT}$, i. e. 
\begin{align*}
    \ff[\widetilde{\cT}] = \bigsqcup_{T \in \widetilde{\cT}} \ff_T.
\end{align*}

Note that $n, m, t, \ff, \cR, \widetilde{\cT}$ meet conditions of Claim \ref{claim: T0 is almost everything}. Thus, due to this claim, there are $(Z_0, x_0) = T_0 \in \widetilde\cT$ and $\delta < 6m^{-1/3}$ such that  $$\big|\ff^*_{T_0}\vee\{T_0\}\big| = (1-\delta)m^{n-t}.$$

It remains to show that in fact $\delta=0$ and $\ff = [m]^n[Z_0\to x_0] = [m]^n[T_0]$. Put $\cR = \ff\setminus\ff[T_0]$. To bound $|\cR|$ we need the following claim. 

\begin{claim}\label{shadows do not intersect}
The families $\partial_{t-1}\ff_{T_0}^*$ and $\partial_{t-1}\left(\cR|_{[n]\setminus Z_0}\right)$ do not intersect.
\end{claim}
\begin{proof}
Assume they do, $T =(Z, x)$ is a common element, $|Z|=t-1$. Fix arbitrary $F \in \ff_{T_0}^*[Z \to x], G \in \cR[Z\to x].$ Put $p = \agr(x_0, G|_{Z_0})$. Since no set from $\cR$ contains $T_0$, we have $p < t-1$. 

Fix $W \subset Z, |W|=t-1-p.$ Let $Y$ be $F|_W$, note that $|Y|\le t-1$. 
By the definition of $\ff_{T_0}^*$ the covering number of $\ff^*_{T_0}(T_0 \cup Y)$ is at least $n+1$ thus there is some $F_1 \in \ff_{T_0}(T_0)[Y]$ that avoids $G$ outside of $Y$. Then, we have $\agr(F_1\sqcup T_0, G) = t-1$, a contradiction.
\end{proof}

Using Theorem \ref{KK_direct}, we get that $$|\partial_{t-1}\ff_{T_0}^*| \ge (1-\delta)^{\frac{t-1}{n-t}}|\partial_{t-1}[m]^{n\setminus Z_0}|.$$ According to the previous claim, using disjointness, $$|\partial_{t-1}\left(\cR_{[n]\setminus Z_0}\right)| \le \left(1 - (1-\delta)^{\frac{t-1}{n-t}}\right)|\partial_{t-1}[m]^{n\setminus Z_0}| \le  \frac{t-1}{n-t}\delta|\partial_{t-1}[m]^{n\setminus Z_0}|.$$ Finally, using Corollary \ref{KK_contra} once again we get $$|\cR| \le m^t\cdot \left|\partial_{n-t}\left(\cR_{[n]\setminus Z_0}\right)\right| \le m^t \cdot \left(\frac{\delta(t-1)}{n-t}\right)^{\frac{n-t}{t-1}}m^{n-t} \le $$ $$\le m^n \cdot m^{-1/3\cdot \frac{n-t}{t-1}} < m^{n-3t}.$$ The last inequality is true because $\frac13\cdot\frac{n-t}{t-1} < 3t$. Finally we apply Lemma \ref{lemma: almost all to all} to $\ff$, $Z_0$ playing the role of $Z$ and $x_0$ playing the role of $x$, which implies $\ff = [m]^n[T_0]$. 


\end{proof}

\section{Proof of the main theorem}
\label{section: proof of the main theorem}



\begin{proof}[Proof of Theorem \ref{theorem: main}]
    Let $\widehat{n}_0(t), \widehat{m}_0(t)$ be polynomials from Corollary \ref{corollary: n > poly(t) ln m} and set $$m_0(t) = \widehat{m}_0(t), \quad \bar{n}_0(t) = \left(88 \ln 2 + 4 \right)t\cdot \widehat{n}_0(t) + 50t^2.$$ Consider integers $n, m, t$ and a family $\mathcal{F} \subset [m]^n$ such that $\agr(x, y) \neq t - 1$ for any $x, y \in \mathcal{F}$ and $|\mathcal{F}| \geq m^{n-t}$. We will prove that if $m \geq m_0(t)$ and $n \geq \bar{n}_0(t)\ln n$, then we may apply either Corollary \ref{corollary: n > poly(t) ln m} or Theorem \ref{m_geq_nt} and conclude $\mathcal{F} = [m]^n[Z\rightarrow x]$ for some $Z \subset n$ such that $|Z|=t$ and $x \in [m]^{Z}$.

    If $n \geq \widehat{n}_0(t)\ln m$, we apply Corollary $\ref{corollary: n > poly(t) ln m}$ and get the desired conclusion. Otherwise, we get
    $$m \geq e^{\frac{n}{\widehat{n}_0(t)}} \geq e^{(88 \ln 2 + 4)t\ln n} \geq e^{(88 \ln 2)t + 4t \ln n} = \left(2^{22}n\right)^{4t}.$$
    Also we have $n \geq \bar{n}_0(t) \geq 50t^2$. Thus, we may apply Theorem \ref{m_geq_nt} and get $\mathcal{F} = [m]^n[Z\rightarrow x]$ for some $Z \subset n$ such that $|Z|=t$ and $x \in [m]^{Z}$.

	Finally, we will show that for $n \ge \widetilde{n}_0(t) = 2^{32} t^3 \ln t$, we have $n \ge \bar{n}_0(t) \ln n$. Indeed, if $n \ge \widetilde{n}_0(t)$, we have
	\begin{align*}
		n & = n^{1 - \frac{1}{3 \ln t}} n^{1/3 \ln t} \ge {\widetilde{n}_0(t)}^{1-1/3 \ln t} \cdot \frac{\ln n}{3 \ln t} \ge 2^{16} t^{3} \ln t \cdot (t^3 \ln t)^{- \frac{1}{3 \ln t}} \cdot \frac{\ln n}{3 \ln t} \\
		&  \ge 2^{16} e^{-4/3}/3 \cdot t^3 \ln n \ge \bar{n}_0(t) \ln n.
	\end{align*}
	In particular, choosing $n_0(t) = 2^{32} t^4$, we get the result.
\end{proof}

\section*{Acknowledgements}

We thank Andrey Kupavskii for his participation in the initial discussions and for his fruitful suggestions on the final draft of this paper.  We are also grateful to Noam Lifshitz for his comments on Theorem~\ref{theorem: hypercontractivity}. Additionally, Fedor Noskov thanks Pierre Youssef for introducing the semigroup method and its widespread applications.

\printbibliography

\appendix

\section{Proof of Lemma \ref{lemma: unbalanced cross matching}}

Lemma \ref{lemma: unbalanced cross matching} is a quantitative analogue of Theorem 2.3 from \cite{keevash2023forbidden} and may be obtained by a verbatim repetition of its proof. We will prove the lemma by contradiction in the following way:

\textbf{Step 1.} Suppose that $\ff_1$ and $\ff_2$ are cross-agreeing. Using a certain transformation, we will transform the families $\ff_i$ into $\widetilde{\ff}_i$ in such a way that they remain cross-agreeing, become monotone, and have large \(p\)-biased measure.

\textbf{Step 2.} Next, we will move from the \(p\)-biased measure to the uniform measure $\mu_{1/2}$ on \(\{0,1\}^n\), and use the monotonicity of $\widetilde{\ff}_i$ to show that $\mu_{1/2}(\widetilde{\ff}_i)$ is sufficiently large.

\textbf{Step 3.} Finally, we will show that $\mu_{1/2}(\widetilde{\ff}_1) + \mu_{1/2}(\widetilde{\ff}_2) > 1$, which leads to a contradiction with cross-agreeing.

\vspace{10pt}

One part of the transformation in Step 1 will be the following compression operator, defined in Section 4.2 of \cite{keevash2023forbidden}.
\begin{definition}
For any $i \in [n]$ and $j \in [m]$, we define the compression operator $T_{i,j} : [m]^n \to [m]^n$. For $x \in [m]^n$, set $T_{i,j}(x) = y \in [m]^n$ where $y_r = x_r$ for all $r \ne i$, and $y_i = x_i$ if $x_i \ne j$, or $y_i = 1$ if $x_i = j$.

We also define a compression operator, also denoted $T_{i,j}$, on codes:
\[
T_{i,j}(\mathcal{F}) = \{x \mid T_{i,j}(x) \in \mathcal{F} \} \cup \{T_{i,j}(x) \mid x \in \mathcal{F} \}.
\]

We define $T_i = T_{i,2} \circ T_{i,3} \circ \cdots \circ T_{i,m}$ for any $i \in [n]$, and $T = T_1 \circ T_2 \circ \cdots \circ T_n$.

We call $\mathcal{F} \subset [m]^n$ \emph{compressed} if $T(\mathcal{F}) = \mathcal{F}$.
\end{definition}

Another part of the transformation in Step 1 will be the transformation from compressed codes in $[m]^n$ to monotone families in the cube $\{0,1\}^n$.

\begin{definition}[Definition 4.2 from \cite{keevash2023forbidden}]
We define $h : [m] \to \{0,1\}$ by $h(1) = 1$ and $h(a) = 0$ for all $a \ne 1$, and $h^{\otimes n} : [m]^n \to \{0,1\}^n$ by $h^{\otimes n}(x) = \big(h(x_1), \ldots, h(x_n)\big)$. 
\end{definition}

To describe the properties of these transformations we need the following definition for $\{0,1\}^n$, which is stricter than that for $[m]^n$.

\begin{definition}
Two families $\mathcal{F}, \mathcal{G} \subset \{0,1\}^n$ are called \emph{cross-agreeing} if
\[
\forall x \in \mathcal{F}, \forall y \in \mathcal{G}, \quad \left| \{ i \in [n] : x_i = y_i = 1 \} \right| \ge 1.
\]
\end{definition}

\begin{definition}
A family $\mathcal{A} \subset \{0,1\}^n$ is called \emph{monotone} if, for any $x, y \in \{0, 1\}^n$ such that $x \le y$ coordinatewise, $x \in \mathcal{A}$ implies $y \in \cA$.
\end{definition}

The transformation at Step 1 will have the following properties, that are direct consequences of Fact 4.1 and Fact 4.3 from \cite{keevash2023forbidden}

\begin{lemma}
\label{lemma:fact-4-1-3-keevash}
Let $\mathcal{F}_1, \mathcal{F}_2 \subset [m]^n$ are cross-agreeing. Let $\widetilde{\mathcal{F}}_i = h^{\otimes n}(T(\mathcal{F}_i))$. Then
\begin{enumerate}
    \item For each $i$ the family $\widetilde{\mathcal{F}}_i$ is monotone and $\mu_{1/m}(\widetilde{\mathcal{F}}_i) \ge \mu(\mathcal{F}_i)$.
    \item $\widetilde{\mathcal{F}}_1$ and $\widetilde{\mathcal{F}}_2$ are cross-agreeing.
\end{enumerate}
\end{lemma}

For Step 2, we will need the following lemma of Ellis, Keller, and Lifshitz~\cite{ellis2019stability} to obtain a lower bound on $\mu_{1/2}$ from $\mu_{1/m}$.

\begin{lemma}
\label{lemma:Ellis-Keller-Lifshitz}
\textit{Suppose $0 \le p \le q \le 1$, $\alpha \ge 0$ and $\mathcal{F} \subset \{0,1\}^n$ is monotone. If $\mu_p(\mathcal{F}) \ge p^\alpha$, then $\mu_q(\mathcal{F}) \ge q^\alpha$.}
\end{lemma}

For Step 3, we will need the following result.
\begin{lemma}[Fact 4.11 from \cite{keevash2023forbidden}]
\label{lemma:fact-4-11-keevash}
If $\mathcal{A}, \mathcal{B} \subset \{0,1\}^n$ are cross-agreeing, then
\[
\mu(\mathcal{A}) + \mu(\mathcal{B}) \le 1.
\]
\end{lemma}


Now we are ready to prove Lemma \ref{lemma: unbalanced cross matching}.

\begin{proof}
    \textbf{Step 1.} Assume, for the sake of contradiction, that the families \(\mathcal{F}_1\) and \(\mathcal{F}_2\) are cross-agreeing. Define the transformed families:
    \[
    \widetilde{\mathcal{F}}_i = h^{\otimes n}(T(\mathcal{F}_i)), \quad \text{for } i = 1, 2.
    \]
    By Lemma \ref{lemma:fact-4-1-3-keevash} the families \(\widetilde{\mathcal{F}}_1\) and \(\widetilde{\mathcal{F}}_2\) are cross-agreeing, monotone and satisfy
    \[
    \mu_{1/m}(\widetilde{\mathcal{F}}_i) \geq \mu(T(\mathcal{F}_i)) = \mu(\mathcal{F}_i), \quad \text{for } i = 1, 2,
    \]
    where \(\mu_p\) denotes the \(p\)-biased measure on \(\{0,1\}^n\).

    \textbf{Step 2.} Now consider \(\widetilde{\mathcal{F}}_i\) under the uniform measure \(\mu_{1/2}\). By monotonicity, we have
    $
    \mu_{1/2}(\widetilde{\mathcal{F}}_1) \geq \mu_{1/m}(\widetilde{\mathcal{F}}_1)
    $. By Lemma \ref{lemma:Ellis-Keller-Lifshitz}, we obtain
    \[
    \mu_{1/2}(\widetilde{\mathcal{F}}_2) \geq \left( \frac{1}{2} \right)^{\log_{1/m} \left( \mu_{1/m}(\widetilde{\mathcal{F}}_2) \right)} = \mu_{1/m}(\widetilde{\mathcal{F}}_2)^{\log_m(2)}.
    \]

    \textbf{Step 3.}
    Combining the above inequalities yields
    \[
    \begin{aligned}
    \mu_{1/2}(\widetilde{\mathcal{F}}_1) + \mu_{1/2}(\widetilde{\mathcal{F}}_2) 
    &\geq \mu_{1/m}(\widetilde{\mathcal{F}}_1) + \mu_{1/m}(\widetilde{\mathcal{F}}_2)^{\log_m(2)} \\
    &\geq \mu(\mathcal{F}_1) + \mu(\mathcal{F}_2)^{\log_m(2)} \\
    &> 1,
    \end{aligned}
    \]
    
    which contradicts Lemma \ref{lemma:fact-4-11-keevash}, completing the proof.
\end{proof}

\end{document}